\documentclass[reqno,12pt]{amsart}
\usepackage[utf8]{inputenc}

\usepackage{graphicx}

\usepackage{amsmath,amssymb,amsthm,mathrsfs,color,times,textcomp,verbatim,yfonts}

\usepackage[left=1.1in, right=1.1in, top=1in]{geometry}

\usepackage[colorlinks=true]{hyperref} 
\hypersetup{urlcolor=blue, citecolor=red, linkcolor=blue}

\usepackage{mathtools}
\mathtoolsset{showonlyrefs,showmanualtags}

\usepackage[square,numbers]{natbib}

\numberwithin{equation}{section}
\theoremstyle{plain} 
\newtheorem{thm}{Theorem}[section]
\newtheorem{prop}[thm]{Proposition}
\newtheorem{lem}[thm]{Lemma}

\theoremstyle{definition}
\newtheorem{defn}[thm]{Definition}

\newtheorem{rem}[thm]{Remark}

\newcommand{\R}{\mathbb{R}}
\renewcommand{\S}{{\mathcal{S}_n}}
\newcommand{\E}{\mathcal{E}}
\newcommand{\sym}{\text{sym}}

\title{Very weak solutions of the Dirichlet problem for 2-Hessian equation}

\author{Tongtong Li}
\address{Institute of Mathematics, Academy of Mathematics and Systems Science, Chinese Academy of Sciences, No.55 Zhongguancun East Road, 100190, Beijing, China}
\email{litongtong221@mails.ucas.ac.cn}

\author{Guohuan Qiu}
\address{Institute of Mathematics, Academy of Mathematics and Systems Science, Chinese Academy of Sciences, No.55 Zhongguancun East Road, 100190, Beijing, China}

\email{qiugh@amss.ac.cn}

\date{\today}

\begin{document}

\maketitle
\begin{abstract}
   For any $\alpha $ small, we construct infinitely many $C^{1,\alpha}$ very weak solutions to the 2-Hessian equation with prescribed boundary value. The proof relies on the convex integration method and cut-off technique.
\end{abstract}

\setlength{\parskip}{0.5em}

\section{Introduction}

Given a smooth function $v:\R^n\to\R$, the $2$-Hessian equation of $v$ is defined by\footnote{Our definition of $\sigma_2$ is different from the standard one by a factor $2$.}
\begin{equation}
\sigma_2(\nabla^2v) :=\sum^n_{i,j=1}[\partial_{ii}v\partial_{jj}v-(\partial_{ij}v)^2] =f. \label{sigma2}
\end{equation}

In the seminal work \cite{caffarelli1985dirichlet}, Caffarelli, Nirenberg and Spruck established the existence and uniqueness of a classical solution to the Dirichlet problem:
\begin{equation}
\begin{cases}\label{eq1}
\sigma_2(\nabla^2 v) = f>0 &\text{in }\Omega\subset\R^n,\ n\ge2, \\
    v = g &\text{on }\partial\Omega,
\end{cases}
\end{equation}
Furthermore, existence and uniqueness of the Dirichlet problem \eqref{eq1} also hold for 2-convex weak solutions in the viscosity sense, see \cite{urbas1990existence,trudinger1990dirichlet}.

The 2-Hessian equation \eqref{sigma2} is a simplified model of the scalar curvature equation for hypersurfaces. Motivated by the Weyl isometric embedding problem, E. Heinz \cite{heinz1959elliptic} established the interior estimates when $n=2$. This is a fully nonlinear partial differential equation for which the regularity is a highly nontrivial problem. Pogorelov \cite{pogorelov1978multidimensional} constructed non-$C^2$ convex solutions to the equation $\det D^{2}u=1$ in $B_{r}\subseteq\mathbb{R}^{n}$ when $n\geq3$. Warren and Yuan \cite{warren2009hessian} established the interior estimate for $\sigma_2(D^2 u)=1$ when $n=3$. For smooth $f>0$ and $n=3$, the second named author \cite{qiu2017interior, qiu2019interior} proved the interior estimates for 2-Hessian equations and scalar curvature equations.
Recently, Shankar and Yuan \cite{shankar2023hessian} derived a priori interior Hessian estimate and interior regularity for $\sigma_2(D^2 u)=1$ in dimension four.  Combining the solvability and uniqueness of the Dirichlet problem \eqref{eq1} for viscosity weak solutions, the above a priori estimates imply that viscosity solutions become smooth when $n=2,3,4$ (at least for $f=1$ when $n=4$). In higher dimensions, additional convexity conditions are assumed to obtain the interior estimate and regularity of the 2-Hessian equation, as seen in papers like \cite{guan2019interior, mcgonagle2019hessian, shankar2020hessian, Mooney2021PAMS,shanker2021regularity}. The question of whether viscosity solutions to \eqref{sigma2} become smooth in dimensions $n \geq 5$ remains a famous open problem in the field. The most crucial ingredient for obtaining the interior estimate is the validity of the Jacobi inequality when $n=3$ and $4$, as demonstrated in \cite{warren2009hessian, qiu2017interior, shankar2023hessian}. But the crucial Jacobi inequality does not hold when $n\geq 5$, see \cite{shankar2023hessian}. Thus one may inquire whether it is possible to construct $C^{1,\alpha}$ viscosity weak solutions by combining the convex integration method with other approaches. The motivation behind our study of \emph{very weak solutions}, as defined in \ref{def1}, is to explore possible ways to construct singular $C^{1,\alpha}$ viscosity weak solutions when $n \geq 5$. Although \emph{very weak solutions} and viscosity solutions are entirely different types of weak solutions, it is still worth exploring possible connections between them.

In this paper, we focus on the non-uniqueness of \emph{very weak solutions}, defined as follows.
\begin{defn}\label{def1}
Let $\Omega$ be a domain in $\R^n$ and $v\in W^{1,2}_{loc}(\Omega)$. We say $v$ solves $\sigma_2(\nabla^2v)=f$ in very weak sense if
\begin{equation*}
-\sum^n_{i,j=1}\int_\Omega \sigma_2^{ij}(\nabla^2\phi)\partial_i v\partial_j v = \int_\Omega f\phi,\quad\forall \phi\in C^\infty_c(\Omega),
\end{equation*}
where
\begin{equation*}
\sigma_2^{ij}(\nabla^2\phi)=\Delta\phi\delta_{ij}-\partial_{ij}\phi.    
\end{equation*}
And $v$ is a very weak solution to \eqref{eq1} if further $v\in C(\bar\Omega)$ and $v=g$ on $\partial\Omega$.
\end{defn}
This definition is motivated by the double divergence structure of $\sigma_2(\nabla^2v)$ (see Section \ref{sec0}). For the Laplace equation $\Delta u=f$, the well-known Weyl lemma guarantees the smoothness of locally integrable weak solutions in the following sense:
\begin{equation*}
\int_\Omega u \Delta \phi = \int_\Omega f \phi, \quad \forall \phi \in C^\infty_c(\Omega).
\end{equation*}
 One interesting point compared with the Laplace equation is that after integrating by parts twice, there are quadratic gradient terms for 2-Hessian equation. For the Monge-Ampere equation, one well-known concept of weak solutions was introduced by Alexandrov \cite{Aleksandrov58}, now commonly referred to as Alexandrov solutions, which are equivalent to viscosity solutions, as shown in \cite{Gutirrez16}.  After Trudinger and Wang's series of challenging works \cite{Trudinger-Wang1, Trudinger-Wang2, Trudinger-Wang3}, Alexandrov solutions have been successfully extended to weak solutions defined by the $2$-Hessian measure for the $\sigma_2$ operator. It is important to note that the concept of the  \emph{very weak} solution is distinct from the weak solution defined in \cite{Trudinger-Wang1,Trudinger-Wang2,Trudinger-Wang3}.

In dimension two, \eqref{sigma2} reduces to Monge-Amp\`ere equation. In this case, the concept of distributional $\sigma_2$ was introduced by \citet{iwaniec01concept} in the name of \emph{weak Hessian}. \citet{lewicka17convex} (see also \cite{jean22revisiting,codenotii2019visualization}) noticed that \eqref{sigma2} is closely related to isometric immersion of surfaces into $\R^3$. Using the convex integration method there, they showed that $C^{1,\alpha}(\bar\Omega)$ very weak solutions to \eqref{sigma2} are dense in $C^0(\bar\Omega)$ for any $\alpha<1/7$ and $f\in L^{7/6}(\Omega)$. This was improved by \citet{cao19very} to $\alpha<1/7$ and recently by \citet*{cao23c1frac13} to $\alpha<1/3$. For the corresponding Dirichlet problem \eqref{eq1}, \citet{cao23dirichlet} proved Theorem \ref{mainthm2} and \ref{mainthm} with $\alpha<1/5$.
 
In general dimensions $n\ge 2$, \citet{lewicka22monge} considered the following Monge-Amp\`ere system
\begin{equation*}
\mathfrak{Det}\nabla^2 v=F\quad\text{in }\Omega\subset\R^n,
\end{equation*} 
where $F=(F_{ij,st})_{1\le i,j,s,t\le n}:\Omega\to\R^{n^4}$ is given, $v:\Omega\to\R^k$ is unknown and
\begin{equation*}
(\mathfrak{Det}\nabla^2 v)_{ij,st} :=\langle \partial_{is}v,\partial_{jt}v\rangle-\langle \partial_{it}v,\partial_{js}v\rangle.
\end{equation*}
She derived the denseness result for $\alpha<\frac{1}{1+(n^2+n)/k}$ (cf. \cite{lewicka23monge2d}). When $k=1$, this covers the counterpart of Theorem \ref{mainthm2} without boundary conditions, since $\sigma_2(\nabla^2 v)=\sum_{i,j}(\mathfrak{Det}\nabla^2v)_{ij,ij}$.
  
In this paper, we prove the following theorems, generalizing the result of \citet{cao23dirichlet} to higher dimensions.
\begin{thm}\label{mainthm2}
    Let $\Omega$ be a bounded domain in $\R^n$ with $C^{2,\kappa}$ boundary $\partial\Omega$, $f\in C^{0,\kappa}(\bar\Omega)$, $\kappa\in (0,1)$. Then there exists some $ g\in C^{2,\kappa}(\partial\Omega)$, such that: for any
    \begin{equation*}
    \alpha < \frac{1}{1+n+n^2},
\end{equation*}
there exist infinitely many very weak solutions $v$ to the Dirichlet problem \eqref{eq1} such that $v\in C^{1,\alpha}(\bar\Omega)$.
\end{thm}
\begin{thm}\label{mainthm}
     Let $\Omega,f,\kappa$ be as in Theorem \eqref{mainthm2}. Assume further $f>0$. Then for any $ g\in C^{0,\kappa}(\partial\Omega)$ and any
    \begin{equation*}
    \alpha < \frac{1}{1+n+n^2},
\end{equation*}
there exist infinitely many very weak solutions $v$ to the Dirichlet problem \eqref{eq1} such that $v\in C^{1,\alpha}(\bar\Omega)$.
\end{thm}

The proof is based on an iteration scheme called convex integration, which was originally introduced by Nash-Kuiper \cite{nash, kuiper} to construct $C^1$ isometric embedding of manifolds, and later generalized (and named) by \citet{gromov74convex, gromov86partial} to more general setting to prove the abundance of (weak) solutions (Gromov's $h$-principle). After Borisov's earlier work \cite{borisov1965isometric,borisov2004irregular}, \citet*{conti12hprinciple,conti12errata} and \citet{cao22global} showed that any $C^1$ differentiable immersion (resp. embedding) of a smooth closed $n$-manifold into $\R^{n+1}$ can be deformed to an $C^{1,\alpha}$ isometric immersion (resp. embedding) if $\alpha<1/(1+n+n^2)$. For surfaces this was improved to $\alpha<1/5$ in \cite{delellis18nash,cao22global}.

These results asserting the \emph{abundance} of (weak) solutions are often referred to as \emph{flexibility} results. On the other hand, one obtains \emph{rigidity}, namely \emph{uniqueness} of solutions, if higher regularity is required. For instance, \citet{borisov1959parallel} showed that any $C^{1,\alpha},\alpha>2/3$ isometric \emph{embedding} of a smooth closed surface with positive Guass curvature into $\R^3$ is unique up to rigid motion.  A modern proof can be found in \cite{conti12hprinciple,conti12errata}. This result is based on the work of \citet{pogorelov1959rigidity,pogorelov1973extrinsic}. \citet[Section 3.5.5.C, Open Problem 34-36]{gromov17geometric} conjectures the regularity threshold of rigidity vs. flexibility to be $\alpha=1/2$. See \cite{delellis2021geometry,pakzad2022convexity} and references therein for recent developments. In \cite{pakzad2022convexity}, Pakzad also proved the rigidity of $C^{1,\alpha}$ very weak solutions to \eqref{eq1} for $n=2,\alpha>2/3$ and $\Omega=\R^2$. One interesting question regarding the rigidity aspect of the 2-Hessian equation is whether a $C^{1,\alpha}$ (for $\alpha \approx 1$) very weak solution of \eqref{eq1} become the viscosity weak solution.


This article is organized as follows. Section \ref{sec0} collects some basic facts, and in Remark \ref{rem1} we explain the idea of the proof. In Section \ref{sec1} we omit the boundary condition $v=g$ and construct very weak solutions by convex integration. The main ingredient is the one stage induction Proposition \ref{prop1}, relying on the decomposition Lemma \ref{lem2}. In Section \ref{sec2} we apply the cut-off technique in \cite{cao23dirichlet} such that the boundary condition is satisfied, thus finishing the proof of Theorem \ref{mainthm2} and \ref{mainthm}.

\section{Preliminaries.}\label{sec0}
\subsection{Distributional \texorpdfstring{$\sigma_2$}{sigma2}}
We observe that $\sigma_2(\nabla^2v)$ has the following divergence structure
\begin{align}
    \sigma_2(\nabla^2 v)
    &=\sum_{i,j}\partial_{i}(\partial_{i}v\partial_{jj}v)-\partial_{j}(\partial_{i}v\partial_{ij}v) \label{div1}\\
    &=\sum_{i,j}\partial_{ij}(\partial_{i}v\partial_{j}v)-\partial_{i}(\partial_{ij}v\partial_{j}v)
    -\partial_{j}(\partial_{i}v\partial_{ij}v) \label{div2}\\
    &=\sum_{i,j}\partial_{ij}(\partial_{i}v\partial_{j}v)
    -\frac{1}{2}\partial_{ii}(\partial_{j}v)^2
    -\frac{1}{2}\partial_{jj}(\partial_{i}v)^2. \label{div3}
\end{align}
By \eqref{div3}, there holds
\begin{align*}
    \int \sigma_2(\nabla^2 v)\phi
    &=\sum_{i,j}\int \partial_{ij}\phi \partial_{i}v\partial_{j}v-\frac{1}{2}\partial_{ii}\phi(\partial_{j}v)^2-\frac{1}{2}\partial_{jj}\phi(\partial_{i}v)^2\\
    &=-\sum_{i,j}\int \sigma_2^{ij}(\nabla^2\phi)\partial_i v\partial_j v,\quad \forall\phi\in C^\infty_c(\R^n),\\
\end{align*}
where
\begin{equation*}
\sigma_2^{ij}(\nabla^2\phi)=\Delta\phi\delta_{ij}-\partial_{ij}\phi.    
\end{equation*}
This allows us to define $\sigma_2(\nabla^2v)$ in distributional sense for any $v\in W^{1,2}_{loc}(\R^n)$. Accordingly, we can define \emph{very weak solutions} to $\sigma_2$ equation as in Definition \ref{def1}.

\begin{rem}
    In fact, we can provide different distributional definitions of $\sigma_2$ based on \eqref{div1} \eqref{div2} \eqref{div3}. Refer to \cite{iwaniec01concept} for a discussion in dimension two.
\end{rem}

\subsection{Notations}
Basically we follow the notations in \cite[Section 2]{cao19very}. The set of $n\times n$ symmetric matrices is denoted by $\S$. We use usual H\"older norms $\|\cdot\|_k,\|\cdot\|_{k,\beta}$. Also we write $\|(v,w)\|_k=\|v\|_k+\|w\|_k$.

We fix a non-negative radial function $\phi\in C^\infty_c(B_1(0))$, such that 
\begin{equation*}
\int_{\R^n}\phi(x)dx=1.
\end{equation*}
This $\phi$ will serve as a mollifier. Namely, let $\phi_l(x)=l^{-n}\phi(\frac{x}{l})$. For any function $h$ on $\R^n$, the mollification of $h$ at length $l$ is 
\begin{equation*}
h*\phi_l(x)=\int_{\R^n} h(x-y)\phi_l(y)dy.
\end{equation*}

A constant is called universal if it depends only on $n,\Omega,f,g$. $C>1$ denotes a positive universal constant that may vary from line to line. Besides, we write $a\lesssim b$ if $a\le Cb$, and $a\sim b$ if $\frac{b}{C}\le a\le Cb$, for some universal constant $C$. 

\subsection{Corrugation functions}
Let
\begin{equation*}
\Gamma_1(s,t)=\frac{s}{\pi}\sin(2\pi t),\quad\Gamma_2(s,t)=-\frac{s^2}{4\pi}\sin(4\pi t)
\end{equation*}
which satisfy the important relation
\begin{equation}\label{relation}
    \partial_t\Gamma_2(s,t)+\frac{1}{2}|\partial_t\Gamma_1(s,t)|^2=s^2.
\end{equation}
It's easy to get following bounds: for $0\le k\le 3$,
\begin{equation}\label{bounds}
\begin{aligned}
    |\partial_t^k\Gamma_1(s,t)|\le Cs,\quad |\partial_s\partial_t^k\Gamma_1(s,t)|\le C,\\
    |\partial_t^k\Gamma_2(s,t)|\le Cs^2,\quad |\partial_s\partial_t^k\Gamma_2(s,t)|\le Cs,\quad |\partial_s^2\partial_t^k\Gamma_2(s,t)|\le C.
\end{aligned}
\end{equation}
\begin{rem}\label{rem1}
    Such periodic functions $\Gamma_1,\Gamma_2$ satisfying \eqref{relation} are crucial in convex integration, whose existence is guaranteed by the quadratic structure in distributional $\sigma_2$. We give some explanations below.
    
    It turns out that \eqref{eq1} is equivalent to \eqref{2} \eqref{4} with $v=g$ on $\partial\Omega$. \eqref{2} is easy to solve since the elements of $A$ are independent. The idea of convex integration is, starting from arbitrary $(v^b,w^b)$, to construct inductively $\{(V_q, W_q)\}_{q=1}^\infty$ that finally converges in $C^{1,\alpha}$ sense to some solution $(v,w)$ of \eqref{4}. 
    
    By linearization of \eqref{4}, the derivative $(\nabla (v-v^b),\nabla(w-w^b)+\sym \nabla(v-v^b)\otimes\nabla v^b))$ has to satisfy a relation similar to $y+\frac{1}{2}x^2=1$, cf. \eqref{relation'}. Since the origin lies in the convex hull of the curve $y+\frac{1}{2}x^2=1$, we let the derivative move along the curve, such that its integration is small. Then $(v,w)$ oscillates around $(v^b,w^b)$, leading to the $C^0$ closeness $\|(v^b-v,w^b-w)\|_0\le\epsilon$. This is where the name "convex integration" comes from. 
    
    This is achieved by adding higher and higher frequency oscillatory functions (i.e. corrugation) produced by $\Gamma_1,\Gamma_2$ at each stage $(V_q,W_q)\to (V_{q+1},W_{q+1})$. Since \eqref{4} is actually a matrix equation, each stage has to be decomposed into $N_*=n(n+1)/2$ steps, which determines $(v,w)\in C^{1,\alpha}$ with $\alpha<1/(1+2N_*)$, cf. \cite{cao19very,cao23c1frac13}.
\end{rem}

\subsection{Mollification estimates.} Recall that for any $r,s\ge0$
\begin{gather}
    [h*\phi_l]_{r}\le [h]_{r}, \label{phil3}\\
    [h-h*\phi_l]_{r}\le C[h]_{r+s}l^{s},\quad [h*\phi_l]_{r+s}\le C[h]_r l^{-s},\label{phil1}\\
    [(h_1h_2)*\phi_l-(h_1*\phi_l)(h_2*\phi_l)]_r\le C\|h_1\|_1\|h_2\|_1 l^{2-r},\label{phil2}
\end{gather}
where $[h]_r=\|\nabla^r h\|_0$. See \cite[Lemma 1]{conti12hprinciple} for the proof. Mollification is important to control higher order derivatives, and the commutator inequality \eqref{phil2} is used to estimate deficit matrices. 

\section{Convex integration.}\label{sec1}
In this section we temporarily omit the boundary condition and construct very weak solutions to
\begin{equation}\label{1}
    \sigma_2(\nabla^2 v) = f \quad\text{in }\Omega.
\end{equation}
We first introduce some notations. Let $L$ be a linear differential operator acting on $\S$-valued functions defined by
\begin{equation}\label{defL}
L(A):=\sum_{i,j}\partial_{ii}A_{jj}+\partial_{jj}A_{ii}-2\partial_{ij}
A_{ij},\quad A=(A_{ij})_{1\le i,j\le n}:\Omega\to\S.
\end{equation}
Then equation \eqref{eq1} can be written as
\begin{equation*}
    -\frac{1}{2}L(\nabla v\otimes\nabla v)=f\quad \text{in}\ \Omega.
\end{equation*}
The operator $L$ generalizes the $\text{curl curl}$ operator in dimension two. In particular, for $\forall$ $w=(w_1,...,w_n):\Omega\to\R^n$, let
\begin{equation*}
    \sym\nabla w:=\frac{1}{2}(\nabla w+\nabla w^T)
    =(\frac{\partial_i w_j+\partial_j w_i}{2})_{1\le i,j\le n},
\end{equation*}
then by simple computation, we can check that
\begin{equation*}
    L(\sym\nabla w)=0
\end{equation*}
in distributional sense.
In fact, due to $\partial_i\sigma_2^{ij}(\nabla^2\phi)=\partial_j\sigma_2^{ij}(\nabla^2\phi)=0$, there holds
    \begin{align*}
        \sum_{i,j}\int \sigma_2^{ij}(\nabla^2\phi)(\frac{\partial_i w_j+\partial_j w_i}{2})
        =-\frac{1}{2}\sum_{i,j}\int \partial_i\sigma_2^{ij}(\nabla^2\phi)w_j
        +\partial_j\sigma_2^{ij}(\nabla^2\phi)w_i
        =0
    \end{align*}
for any $\phi\in C_c^\infty(\Omega)$.

So to solve \eqref{1} it suffices to find $(A,v,w)\in C^0(\bar\Omega,\S)\times C^1(\bar\Omega)\times C^1(\bar\Omega, \R^n)$ such that
\begin{align}
    -L(A)=f\quad\text{in}\ \Omega, \label{2}\\
    A=\frac{1}{2}\nabla v\otimes\nabla v+\sym\nabla w\quad\text{in}\ \Omega.\label{4}
\end{align}

\textbf{Part one: solve \eqref{2}.}
Let $A=(u+\tau)Id$, where $\tau>0$ is a constant and $u$ solves the Dirichlet problem
\begin{equation*}
-(2n-2)\Delta u =f\ \text{in}\ \Omega,\quad u=0\ \text{on}\ \partial\Omega.
\end{equation*}
Then
\begin{equation*}
    -L(A)=-(2n-2)\Delta u=f
\end{equation*}
and $A\in C^{2,\kappa}(\bar\Omega,\S)$ by classical elliptic PDE theory.

\textbf{Part two: solve \eqref{4}.}
Replacing $v^b$ by its mollification, we may assume $v^b\in C^\infty(\bar\Omega)$, and let $w^b=0$. Denote the initial deficit matrix by
\begin{equation}\label{901}
D^b:=A-\frac{1}{2}\nabla v^b\otimes\nabla v^b-\sym\nabla w^b.
\end{equation}
Recall $A=(u+\tau)Id$. By letting $\tau$ large enough, the problem boils down to the following:
\begin{thm}[\cite{lewicka22monge}, Theorem 1.1, $(d,k)=(n,1)$]\label{thm1}
Let $\Omega\subset\mathbb R^n$ be an open, bounded domain. Given a function $v^b\in C^1(\bar\Omega)$ a vector field $w^b\in C^1(\bar\Omega, \mathbb R^n)$ and a matrix field $A\in C^{0,\beta}(\bar\Omega, \S)$, assume that: 
\begin{equation}\label{902}
D^b:=A-\frac{1}{2}\nabla v^b\otimes\nabla v^b-\sym\nabla w^b\quad \text{satisfies}\quad D^b>\bar \tau Id\ on\ \bar\Omega,
\end{equation} 
for some $\bar \tau>0$, in the sense of matrix inequalities. Fix $\epsilon>0$ and let: 
\begin{equation*}
0< \alpha < \min\{\frac{\beta}{2},\frac{1}{1+n+n^2}\}.
\end{equation*} 
Then, there exists $v\in C^{1,\alpha}(\bar\Omega)$ and $w\in C^{1,\alpha}(\bar\Omega,\R^n)$ such that the following holds: \begin{gather}
    \|v-v^b\|_0\le\epsilon,\quad \|w-w^b\|_0\le\epsilon, \label{thm1.1}\\
    A=\frac{1}{2}\nabla v\otimes\nabla v+\sym\nabla w\quad \text{in }\Omega. \label{thm1.2}
\end{gather}
\end{thm}
 For our purpose we only prove for $A\in C^{2,\kappa}(\bar\Omega,\S)$, and w.l.o.g we assume  $\bar\tau>1$. We divide its proof into four steps following \cite{cao19very}. The key is the one stage induction Proposition \ref{prop1} in Step 3.
 
\begin{rem}\label{rem999}
    Step 1-2 are preparations in nature,  both consisting of inductive constructions that are very similar to (and in fact easier than) Step 3. Besides, after finishing this note, we learned a trick from \cite[Section 4]{cao23c1frac13} that helps to bypass Step 1-2. So the readers may directly skip to Step 3. The trick goes as follows. Recall $A=(u+\tau)Id$. Fix
    \[
    \tau = (K+\sigma^{-1})(\|u\|_0+\|v^b\|_2^2+100),
    \]
    and let
    \[
    \bar A=\delta_1\tau^{-1} A,\quad V_0 = \delta_1^{\frac{1}{2}}\tau^{-1/2} v^b,\quad W_0 = 0.
    \]
    Then, replacing $A$ by $\bar A$, the assumptions \eqref{prop1.a1}-\eqref{prop1.a3} hold at $q=0$. By induction we get a solution $(v,w)$ to \eqref{1} with $A$ replaced by $\bar A$. Then $(\delta_1^{-\frac{1}{2}}\tau^{1/2}v,\delta_1^{-1}\tau w)$ is a solution to \eqref{1}.
\end{rem}

\textbf{Step 1 (Preparation): $(v^b,w^b)\to(\bar v_N,\bar w_N)$.}
In this step we aim to reduce $D^b$ to the neighborhood of identity matrix $Id$, i.e. to construct $(\bar v_i,\bar w_i),1\le i\le N$ inductively such that
\begin{gather}
    \|(\bar v_N,\bar w_N)-(v^b,w^b)\|_{0}\le \frac{\epsilon}{3}, \label{barvN}\\
    \|\bar D\|_0\le\frac{\sigma_*}{3} \label{barD},
\end{gather}
where $\sigma_*>0$ is given as in Lemma \ref{lem2} and
\begin{equation}
\bar D:=A-\frac{1}{2}\nabla\bar v_N\otimes\nabla\bar v_N-\sym\nabla\bar w_N-Id.
\end{equation}
Note that by \eqref{901}
\begin{equation}\label{903}
    \bar D=(D^b-Id) + (\frac{1}{2}\nabla v^b\otimes\nabla v^b+\sym\nabla w^b)-(\frac{1}{2}\nabla\bar v_N\otimes\nabla\bar v_N+\sym\nabla\bar w_N)
\end{equation}
The construction relies on the following lemma \ref{lem1}.
\begin{lem}[\cite{nash}, Lemma 1]\label{lem1}
    Given two constants $\Lambda>\lambda>0$, there exists constants $\Lambda_*>\lambda_*\ge 0$, a positive integer $N$, constant unit vectors $ \bar{\xi}_1,..., \bar{\xi}_N$ in $\R^n$ and $C^\infty$ functions $\bar d_1,...\bar d_N$ from $S_n$ to $\R$, depending only on $n,\lambda,\Lambda$, such that for any $D\in\S$ and
    \begin{equation*}
    \lambda Id<D<\Lambda Id
    \end{equation*}
    we have
    \begin{equation*}
    D=\sum_{i=1}^{N}\bar d_i^2 \bar{\xi}_i\otimes \bar{\xi}_i,\  \bar{d}_i= \bar{d}_i(D)\in [\lambda_*,\Lambda_*]\ \forall i.
    \end{equation*}
\end{lem}
\begin{proof}[Proof of \eqref{barvN}-\eqref{barD}.]
    Because of \eqref{902} and $\bar\tau>1$, we can apply  Lemma \ref{lem1} to $D^b-Id$, which yields
\begin{equation}\label{905}
D^b-Id=\sum_{i=1}^{N}\bar a_i^2 \bar{\xi}_i\otimes \bar{\xi}_i,\ \bar a_i\in [\lambda_*,\Lambda_*].
\end{equation}
Now we define, for $1\le i\le N$, 
\begin{align}
    &(\bar v_0,\bar w_0)=(v^b,w^b),\label{bar1}\\
    &\bar v_i=\bar v_{i-1}+\frac{1}{ \bar{\mu}_i}\Gamma_1( \bar{a}_i, \bar{\mu}_i x\cdot \bar{\xi}_i),\label{bar2}\\
    &\bar w_i=\bar w_{i-1}-\frac{1}{ \bar{\mu}_i}\Gamma_1( \bar{a}_i, \bar{\mu}_i x\cdot \bar{\xi}_i)\nabla\bar v_{i-1}+\frac{1}{ \bar{\mu}_i}\Gamma_2( \bar{a}_i, \bar{\mu}_i x\cdot \bar{\xi}_i) \bar{\xi}_i,\label{bar3}
\end{align}
where the frequencies $ \bar{\mu}_i>1$ are constants determined later. It will be clear later that the $i$-th step $(\bar v_{i-1},\bar w_{i-1})\to (\bar v_i,\bar w_i)$ cancels the term $ \bar{a}_i^2 \bar{\xi}_i\otimes \bar{\xi}_i$. Denote the $i$-th step error by
\begin{equation*}
    \bar{\E}_i:= \bar{a}_i^2 \bar{\xi}_i\otimes \bar{\xi}_i
    +(\frac{1}{2}\nabla\bar v_{i-1}\otimes\nabla\bar v_{i-1}+\sym\nabla \bar w_{i-1})
    -(\frac{1}{2}\nabla\bar v_i\otimes\nabla\bar v_i+\sym\nabla \bar w_i).
\end{equation*}
Then, combining \eqref{903} and \eqref{905}
\begin{equation*}
\bar D=\sum_{i=1}^{N}\bar\E_i.
\end{equation*}
We now estimate $\bar{\E}_i$. Using \eqref{bar2} and \eqref{bar3}, we get
\begin{align*}
\nabla\bar v_i&=\nabla\bar v_{i-1}+\frac{1}{\bar\mu_i}\nabla\Gamma_1,\\
\nabla\bar w_i&=\nabla\bar w_{i-1}-\frac{1}{\bar\mu_i}\nabla\Gamma_1\otimes\nabla\bar v_{i-1}-\frac{1}{\bar\mu_i}\Gamma_1\nabla^2\bar v_{i-1}+\frac{1}{\bar\mu_i}\nabla\Gamma_2\otimes\bar\xi_i,
\end{align*}
Here we write $\Gamma_1=\Gamma_1(\bar a_i,\bar\mu_i x\cdot\bar\xi_i),\Gamma_2=\Gamma_2(\bar a_i,\bar\mu_i x\cdot\bar\xi_i)$, hence
\begin{align*}
    \nabla\Gamma_1=(\partial_{s}\Gamma_1)\nabla\bar{a}_i+\bar{\mu}_i(\partial_{t}\Gamma_1) \bar{\xi}_i,\quad
    \nabla\Gamma_2=(\partial_{s}\Gamma_2)\nabla\bar{a}_i+\bar{\mu}_i(\partial_{t}\Gamma_2) \bar{\xi}_i.
\end{align*}
It follows that
\begin{align*}
    &(\frac{1}{2}\nabla\bar v_i\otimes\nabla\bar v_i+\sym\nabla\bar w_i)-(\frac{1}{2}\nabla\bar v_{i-1}\otimes\nabla\bar v_{i-1}+\sym\nabla\bar w_{i-1})\\
    &=-\frac{1}{ \bar{\mu}_i}\Gamma_1\nabla^2\bar v_{i-1}
    +\frac{1}{ \bar{\mu}_i}\sym(\nabla\Gamma_2\otimes \bar{\xi}_i)
    +\frac{1}{2 \bar{\mu}_i^2}\nabla\Gamma_1\otimes\nabla\Gamma_1\\
    &= -\frac{1}{ \bar{\mu}_i}\Gamma_1\nabla^2\bar v_{i-1}
    +\frac{1}{ \bar{\mu}_i}[\partial_{s}\Gamma_2+(\partial_{s}\Gamma_1)(\partial_{t}\Gamma_1)]\sym(\nabla  \bar{a}_i\otimes \bar{\xi}_i)\\
    &\quad +(\partial_{t}\Gamma_2+\frac{1}{2}|\partial_{t}\Gamma_1|^2) \bar{\xi}_i\otimes \bar{\xi}_i+\frac{1}{2 \bar{\mu}_i^2}|\partial_{s}\Gamma_1|^2\nabla  \bar{a}_i\otimes\nabla  \bar{a}_i
\end{align*}
Recalling \eqref{relation}, we have
\begin{equation}\label{relation'}
\partial_{t}\Gamma_2+\frac{1}{2}|\partial_{t}\Gamma_1|^2=\bar a_i^2,
\end{equation}
therefore
\begin{equation}
\begin{aligned}
    \bar\E_i
    &=\frac{1}{ \bar{\mu}_i}\Gamma_1\nabla^2\bar v_{i-1}
    -\frac{1}{ \bar{\mu}_i}[\partial_{s}\Gamma_2+(\partial_{s}\Gamma_1)(\partial_{t}\Gamma_1)]\sym(\nabla  \bar{a}_i\otimes \bar{\xi}_i)\\
    &\quad -\frac{1}{2 \bar{\mu}_i^2}|\partial_{s}\Gamma_1|^2\nabla  \bar{a}_i\otimes\nabla  \bar{a}_i\\
\end{aligned}
\end{equation}
Using \eqref{bounds}, it is easy to see
\begin{align*}
    \|\bar\E_i\|_0\le\frac{C(\bar v_{i-1}, \bar{a}_i)}{ \bar{\mu}_i},\\
    \|(\bar v_i-\bar v_{i-1},\bar w_i-\bar w_{i-1})\|_0\le\frac{C(\bar v_{i-1}, \bar{a}_i)}{ \bar{\mu}_i}.
\end{align*}
Choosing $\bar{\mu}_i$ large enough inductively, we can obtain
\begin{equation*}
    \|\bar{\E}_i\|_0\le \frac{\sigma_*}{3N},\quad
    \|(\bar v_i-\bar v_{i-1},\bar w_i-\bar w_{i-1})\|_0\le\frac{\epsilon}{3N}
\end{equation*}
where $\sigma_*>0$ is given in Lemma \ref{lem2} below. As a result
\begin{equation}\label{1002}
    \|\bar D\|_0\le\frac{\sigma_*}{3},\quad
    \|(\bar v_N-v^b,\bar w_N-w^b)\|_0\le\frac{\epsilon}{3}.
\end{equation}
\end{proof}
\textbf{Step 2 ($0$th approximate): $(\bar v_N,\bar w_N)\to(V_0,W_0)$}.
In this step we construct $0$th stage approximate solution $(V_0,W_0)$ such that the induction assumptions in Proposition \ref{prop1} holds at $q=0$. This relies on the following decomposition. The proof can be found in \cite[Lemma 5.2]{conti12hprinciple}. 
\begin{lem}\label{lem2}
    Let $N_*=\frac{1}{2}n(n+1)$ be the dimension of $\S$. Then there exist constants $\sigma_*\in(0,1/2), C_*>c_*>0$, unit vectors $ \xi_1,..., \xi_{N_*}$ in $\R^n$ and $C^\infty$ functions $d_1,...,d_{N_*}$ from $\S$ to $\R$, depending only on $n$, such that for any $D\in\S$ and
    \begin{equation*}
	\|D-Id\|_0\le\sigma_*,
    \end{equation*}
    we have
    \begin{equation*}
        D=\sum_{i=1}^{N_*}d_i^2 \xi_i\otimes \xi_i,\ d_i=d_i(D)\in[c_*,C_*],\ \forall i.
    \end{equation*}
\end{lem}

We construct $(V_0,W_0)$ as follows. Let
\begin{equation}\label{ineq01}
    \delta_1\le\frac{\sigma_*}{3}
\end{equation}
be a constant given by $\delta_1=a^{-b}$ as in Step 3. Then by \eqref{1002}
\begin{equation*}
    \|\bar D-\delta_1 Id\|_0\le\sigma_*,
\end{equation*}
thus by Lemma \ref{lem2}
\begin{equation*}
    \bar D-\delta_1 Id=\sum_{i=1}^{N_*}\hat{a}_i^2 \xi_i\otimes \xi_i,
\end{equation*}
where 
\begin{equation}
    \|\hat a_i\|_k\lesssim 1,\ \forall 0\le k\le 3.
\end{equation}
with implicit constants depending on $v^b,w^b$. Similar to Step 1, we define
\begin{align}
    &( \hat v_0, \hat w_0)=(\bar v_N, \bar w_N),\\
    & \hat{v}_i= \hat{v}_{i-1}+\frac{1}{\hat{\mu}_i}\Gamma_1(\hat{a}_i, \hat{\mu}_i x\cdot \xi_i),\\
    & \hat{w}_i= \hat{w}_{i-1}-\frac{1}{\hat{\mu}_i}\Gamma_1(\hat{a}_i,\hat{\mu}_i x\cdot \xi_i)\nabla \hat{v}_{i-1}+\frac{1}{\hat{\mu}_i}\Gamma_2(\hat{a}_i,\hat{\mu}_i x\cdot \xi_i) \xi_i, \\
    &(V_0,W_0)=(\hat v_{N_*},\hat w_{N_*}),
\end{align}
with constants $1\ll\hat{\mu}_1<...<\hat{\mu}_{N_*}$ determined later. The $i$-th step error is
\begin{equation*}
\begin{aligned}
    \hat{\E}_i
    &:=\hat{a}_i^2 \xi_i\otimes \xi_i
    +(\frac{1}{2}\nabla  \hat{v}_{i-1}\otimes\nabla  \hat{v}_{i-1}+\sym\nabla   \hat{w}_{i-1})
    -(\frac{1}{2}\nabla  \hat{v}_i\otimes\nabla  \hat{v}_i+\sym\nabla   \hat{w}_i)\\
    & =\frac{1}{ \hat{\mu}_i}\Gamma_1\nabla^2  \hat{v}_{i-1}
    -\frac{1}{ \hat{\mu}_i}[\partial_{s}\Gamma_2+(\partial_{s}\Gamma_1)(\partial_{t}\Gamma_1)]\sym(\nabla \hat{a}_i\otimes \xi_i)\\
    & -\frac{1}{2\hat{\mu}_i^2}|\partial_{s}\Gamma_1|^2\nabla \hat{a}_i\otimes\nabla \hat{a}_i,
\end{aligned}
\end{equation*}
and by \eqref{881}
\begin{equation*}
     D_0=\sum_{i=1}^{N_*}\hat\E_i.
\end{equation*}
Unlike before, we estimate $\hat{\E}_i$ more carefully.
One can see that
\begin{align*}
    \|\nabla^k \hat{v}_i-\nabla^k \hat{v}_{i-1}\|_0\lesssim\hat{\mu}_i^{k-1},\ 0\le k\le 3,\\
    \|\nabla^k \hat{w}_i-\nabla^k \hat{w}_{i-1}\|_0\lesssim\hat{\mu}_i^{k-1},\ 0\le k\le 2.
\end{align*}
Besides $\|(\hat{v}_0,\hat{w}_0)\|_3\le C$ for some $C$ depending on $v^b,w^b$. Thus
\begin{gather*}
\|(\hat{v}_i,\hat{w}_i)\|_1\lesssim 1,\\
\|(\hat{v}_i,\hat{w}_i)\|_2\lesssim 1+\hat{\mu}_1+...+\hat{\mu}_i\lesssim\hat{\mu}_i.
\end{gather*}
Hence we get
\begin{equation*}\label{hatEi}
    \|\hat{\E}_i\|_0
    \lesssim\frac{1}{\hat{\mu}_i}\|\nabla^2 \hat{v}_{i-1}\|_0+\frac{1}{\hat{\mu}_i}+\frac{1}{\hat{\mu}_i^2}
    \lesssim \frac{\hat{\mu}_{i-1}}{\hat{\mu}_i}
\end{equation*}
where we write $\hat{\mu}_0=1$. It follows that
\begin{gather*}
    \| D_0\|_0\le\sum_{i=1}^{N_*}\|\hat\E_i\|_0\lesssim \frac{1}{\hat{\mu}_1}+\frac{\hat{\mu}_1}{\hat{\mu}_2}+...+\frac{\hat{\mu}_{N_*-1}}{\hat{\mu}_{N_*}},\\
    \|(V_0,W_0)\|_2\lesssim \hat{\mu}_{N_*}.
\end{gather*}
We are ready to determine $\hat{\mu}_i$. Let
\begin{equation*}
    \hat{\mu}_i:=(\frac{C}{\sigma\delta_1})^i.
\end{equation*}
By choosing $C>1$ large enough, and then $K>1$ large enough, we get
\begin{gather}
    \| D_0\|_0\le {\sigma}\delta_1,\\
    \|(V_0,W_0)\|_2\le C(\sigma\delta_1)^{-N_*}\le K\delta_1^{-N_*}.\label{K3}
\end{gather}
where $\sigma>0$ and $K>1$ are universal constants as in Proposition \ref{prop1}. We now require
\begin{equation}\label{ineq*}
    \delta_1^{-N_*}\le \delta_0^\frac{1}{2}\lambda_0,
\end{equation}
hence
\begin{gather}
    \| D_0\|_0\le {\sigma}\delta_1, \label{011}\\
    \|(V_0,W_0)\|_2\le K\delta_0^\frac{1}{2}\lambda_0.\label{012}
\end{gather}
Besides, if $K>1$ is large enough
\begin{align}
    &\|(V_0-\bar v_N,W_0-\bar w_N)\|_0\lesssim \frac{1}{\hat\mu_1}+...+\frac{1}{\hat\mu_N}\le\sigma\delta_1,\label{013}\\
    &\|(V_0,W_0)\|_1=\|(\hat v_{N_*},\hat w_{N_*})\|_1\lesssim 1\le \frac{\sqrt{K}}{2}\label{014}.
\end{align}
\textbf{Step 3 (Induction): $(V_q,W_q)\to( V_{q+1},W_{q+1})$.}
We are ready to do induction. Set $a>1,1<b<2,c>0$ as parameters. Take a sequence of (small) amplitudes
\begin{equation}
    \delta_q:=a^{-b^q}
\end{equation}
and a sequence of (large) frequencies
\begin{equation}
    \lambda_q:=a^{cb^q}
\end{equation}
Denote the $q$-th stage deficit matrix by
\begin{equation}\label{881}
   D_q:= A-\frac{1}{2}\nabla V_q\otimes\nabla V_q-\sym\nabla W_q-\delta_{q+1}Id.
\end{equation}
We need the following one stage induction:
\begin{prop}\label{prop1}
    There exist some positive universal constants $\sigma\in(0,\sigma_*/3)$ and $K>1$, such that: if $(V_q,W_q)\in C^2(\bar\Omega)\times C^2(\bar\Omega,\R^n)$ satisfy
    \begin{gather}
         \|(V_q,W_q)\|_1\le \sqrt{K}, \label{prop1.a1}\\
         \|(V_q,W_q)\|_2\le K\delta_q^\frac{1}{2}\lambda_q, \label{prop1.a2}\\
         \|D_q\|_0\le{\sigma}\delta_{q+1}. \label{prop1.a3}
    \end{gather}
    we can construct $(V_{q+1},W_{q+1})\in C^2(\bar\Omega)\times C^2(\bar\Omega,\R^n)$ with
    \begin{gather}
        \|(V_{q+1}-V_q, W_{q+1}-W_q)\|_1\le K\delta_{q+1}^\frac{1}{2}, \label{prop1.0}\\
        \|(V_{q+1},W_{q+1})\|_2\le K\delta_{q+1}^\frac{1}{2}\lambda_{q+1},\label{prop1.1}\\
         \|D_{q+1}\|_0\le{\sigma}\delta_{q+2}.\label{prop1.2}
    \end{gather}
\end{prop} 
\begin{proof}
In the proof all implicit (universal) constants are independent of $\sigma,K$. For convenience we write
\begin{equation}\label{mu0}
    \mu_0:=K\delta_{q+1}^{-\frac{1}{2}}\delta_q^\frac{1}{2}\lambda_q,
\end{equation}
thus
\begin{equation}\label{prop1.a2'}
    \|(V_q,W_q)\|_2\le \delta_{q+1}^\frac{1}{2}\mu_0.
\end{equation}
Let $\tilde A=A*\phi_l,\tilde v=V_q*\phi_l,\tilde w=W_q*\phi_l$ be the mollification of $A,V_q,W_q$, respectively, at length $l$. 
\begin{rem}
    Unlike before, here we need mollification to avoid the dependence on uncontrolled higher derivatives of $(V_q,W_q)$, see Remark \ref{rem998}. Mollification is usually done in the entire space $\R^n$. To this end we actually first extend $v^b,f$ to $C_c(\tilde\Omega)$ for some $\tilde\Omega\Supset\Omega$. and then solve \eqref{1} in $\tilde\Omega$. This does not affect the arguments in the next section, since we cut-off everything there inside $\Omega$.

\end{rem}
Denote
\begin{equation*}
    \tilde D := \tilde A-\frac{1}{2} \nabla\tilde v\otimes\nabla\tilde v-\sym\nabla\tilde w-\delta_{q+2}Id.
\end{equation*}
Then by \eqref{881}
\begin{equation*}
    \tilde D-\delta_{q+1}Id= D_q*\phi_l+\frac{1}{2}[(\nabla V_q\otimes\nabla V_q)*\phi_l-\nabla\tilde v\otimes\nabla\tilde v]-\delta_{q+2}Id.
\end{equation*}
Assume
\begin{equation}\label{ineq02}
    \delta_{q+2}\le {\sigma}\delta_{q+1},
\end{equation}
and injecting \eqref{phil3} \eqref{phil2} and assumptions \eqref{prop1.a1} \eqref{prop1.a2'} \eqref{prop1.a3}, we see
\begin{align*}
    \|\tilde D-\delta_{q+1}Id\|_0
    &\le \|D_q*\phi_l\|_0+\frac{1}{2}\|(\nabla V_q\otimes\nabla V_q)*\phi_l-\nabla\tilde v\otimes\nabla\tilde v\|_0+\delta_{q+2}\\
    &\le \|D_q\|_0 + Cl^2\|V_q\|_2^2 + \delta_{q+2}\\
    &\le 2{\sigma}\delta_{q+1}+C(\mu_0 l)^2\delta_{q+1}
\end{align*}
Now we take
\begin{equation}\label{l}
    l:=\frac{\sigma}{C\mu_0} = \frac{\sigma}{CK}\delta_{q+1}^\frac{1}{2}\delta_q^{-\frac{1}{2}}\lambda_q^{-1}
\end{equation}
for some constant $C>1$ large enough, independent of $\sigma$ and $K$,  
such that
\begin{equation*}
    \|\tilde D-\delta_{q+1}Id\|_0\le 3\sigma\delta_{q+1},
    \quad\text{i.e.}\quad
    \|\frac{\tilde D}{\delta_{q+1}}-Id\|_0\le 3\sigma\le\sigma_*.
\end{equation*}
Similarly 
\begin{equation*}
    \|\tilde D-\delta_{q+1}Id\|_k \le C\sigma\delta_{q+1}l^{-k},\quad\forall 1\le k\le 3.
\end{equation*}
Now Lemma \ref{lem2} applies to $\frac{\tilde D}{\delta_{q+1}}$, yielding
\begin{equation}\label{882}
    \tilde D=\sum_{i=1}^{N_*}a_i^2\xi_i\otimes\xi_i,\quad\text{i.e.}\quad
    \tilde A-\frac{1}{2} \nabla\tilde v\otimes\nabla\tilde v-\sym\nabla\tilde w-\delta_{q+2}Id=\sum_{i=1}^{N_*}a_i^2\xi_i\otimes\xi_i,
\end{equation}
where
\begin{gather}
    \|a_i\|_0\lesssim\delta_{q+1}^\frac{1}{2}, \label{721}\\
    \|\nabla^k a_i\|_0\lesssim\delta_{q+1}^\frac{1}{2}\|\frac{\tilde D}{\delta_{q+1}}-Id\|_k\lesssim \sigma\delta_{q+1}^\frac{1}{2}l^{-k},\ \forall 1\le k\le 3. \label{722}
\end{gather}
Now we define, for $1\le i\le N_*$
\begin{align}
    &(v_0, w_0)=(\tilde v, \tilde w)\\
    & v_i= v_{i-1}+\frac{1}{\mu_i}\Gamma_1(a_i, \mu_i x\cdot \xi_i),\\
    & w_i= w_{i-1}-\frac{1}{\mu_i}\Gamma_1(a_i,\mu_i x\cdot \xi_i)\nabla v_{i-1}+\frac{1}{\mu_i}\Gamma_2(a_i,\mu_i x\cdot \xi_i) \xi_i,\\
    & (V_{q+1},W_{q+1})=(v_{N_*},w_{N_*}),
\end{align}
with constants \begin{equation*}
\mu_0\le\mu_1\le...\le\mu_{N_*}=\lambda_{q+1}
\end{equation*}
determined later. Using \eqref{phil1} \eqref{phil2} and \eqref{prop1.a2'}, we see
\begin{gather}
    \|(v_0-V_q,w_0-W_q)\|_1\lesssim \|(V_q.W_q)\|_2 l\lesssim \delta_{q+1}^\frac{1}{2}\mu_0 l\lesssim \sigma\delta_{q+1}^\frac{1}{2},\label{401}\\
    \|v_0\|_{2+k}\lesssim \|V_q\|_2 l^{-k}\lesssim \delta_{q+1}^\frac{1}{2}\mu_0 l^{-k},\ k=0,1. \label{402}
\end{gather}
We first estimate $v_i$ as follows: by  \eqref{bounds} and \eqref{721} \eqref{722}, for $0\le k\le 3$ there holds
\begin{equation}\label{403}
\begin{aligned}
    \|v_i- v_{i-1}\|_k
    &\lesssim \frac{1}{\mu_i}(\mu_i^k\|a_i\|_0+\mu_i^{k-1}\|a_i\|_1+...+\|a_i\|_k)\\
    &\lesssim\frac{1}{\mu_i}(\mu_i^k\delta_{q+1}^\frac{1}{2}+\mu_i^{k-1}\cdot\sigma\delta_{q+1}^\frac{1}{2}l^{-1}+...+\sigma\delta_{q+1}^\frac{1}{2}l^{-k})\\
    &\lesssim \delta_{q+1}^\frac{1}{2}\mu_i^{k-1}
\end{aligned}
\end{equation}
where we assumed $l^{-1}\le\mu_1$, namely
\begin{equation}\label{ineq03}
    \frac{C\mu_0}{\sigma}\le\mu_1\le...\le\mu_{N_*}=\lambda_{q+1}.
\end{equation}
As a byproduct
\begin{gather*}
    \|\nabla v_i\|_0\le \|\nabla V_q\|_0+C\delta_{q+1}^\frac{1}{2}\le \sqrt{K}+C\lesssim\sqrt{K},\\
    \|\nabla^2 v_i\|_0\lesssim \delta_{q+1}^\frac{1}{2}(\mu_0+\mu_1+...+\mu_i)\lesssim\delta_{q+1}^\frac{1}{2}\mu_i,\\
    \|\nabla^3 v_0\|_0\lesssim \delta_{q+1}^\frac{1}{2}\mu_0 l^{-1},\\
    \|\nabla^3 v_i\|_0\lesssim \delta_{q+1}^\frac{1}{2}\mu_0 l^{-1}+\delta_{q+1}^\frac{1}{2}(\mu_1^2+...+\mu_i^2)\lesssim \delta_{q+1}^\frac{1}{2}\mu_i^2.
\end{gather*}
These will be used to estimate $w_i$ and $\E_i$.

Now turn to $w_i$. By \eqref{bounds} \eqref{721} \eqref{722} and above bounds for $v_{i-1}$,  we have
\begin{equation}\label{404}
\begin{aligned}
    \|w_i-w_{i-1}\|_1
    &\lesssim \frac{1}{\mu_i}(\|a_i\|_0\|\nabla v_{i-1}\|_0\mu_i+\|a_i\|_1\|\nabla v_{i-1}\|_0+\|a_i\|_0\|\nabla v_{i-1}\|_1)\\
    &\quad +\frac{1}{\mu_i}(\|a_i^2\|_0\mu_i+\|a_i^2\|_1)\\
    &\lesssim  \frac{1}{\mu_i}(\delta_{q+1}^\frac{1}{2} \cdot\sqrt{K}\mu_i+\delta_{q+1}^\frac{1}{2}l^{-1}\cdot\sqrt{K}+\delta_{q+1}^\frac{1}{2}\cdot \delta_{q+1}^\frac{1}{2}\mu_{i-1})\\
    &\quad +\frac{1}{\mu_i}(\delta_{q+1}\mu_i+\delta_{q+1}l^{-1})\\
    &\lesssim\sqrt{K}\delta_{q+1}^\frac{1}{2}.
\end{aligned}
\end{equation}
Similarly 
\begin{equation}\label{405}
    \|w_i-w_{i-1}\|_2\lesssim \sqrt{K}\delta_{q+1}^\frac{1}{2}\mu_i.
\end{equation}
\begin{rem}\label{rem998}
Although to solve \eqref{4} only requires $W_q$ converges in $C^1$ sense, \eqref{405} is necessary: note that to estimate $\|\nabla^2 a_i\|_0$ used in \eqref{403}, we need to mollify $W_q$, which produces an error $\|w_0-W_q\|_1$, hence the bound of $\|W_q\|_2$, in particular \eqref{405}, is unavoidable. 
\end{rem}
By \eqref{401} \eqref{403} \eqref{404}
\begin{equation}\label{K6}
\begin{aligned}
    &\|(V_{q+1}-V_q, W_{q+1}-W_q)\|_1\\
    &\lesssim \|(v_0,w_0)-(V_q,W_q)\|_1+\sum_{i=1}^{N_*}(\|v_i-v_{i-1}\|_1+\|w_i-w_{i-1}\|_1)\\
    &\lesssim\sqrt{K}\delta_{q+1}^\frac{1}{2}
    \le K\delta_{q+1}^\frac{1}{2},
\end{aligned}
\end{equation}
and similarly by \eqref{402} \eqref{403} \eqref{405}
\begin{equation}\label{K2}
    \|(V_{q+1}-V_q, W_{q+1}-W_q)\|_2\lesssim \sqrt{K}\delta_{q+1}^\frac{1}{2}(\mu_0+...+\mu_{N_*})
    \le K\delta_{q+1}^\frac{1}{2}\lambda_{q+1},
\end{equation}
as long as $K>1$ is large enough. This verifies \eqref{prop1.0} \eqref{prop1.1}. 

It remains to prove \eqref{prop1.2}. The $i$-th step error is
\begin{align*}
    {\E}_i
    &={a}_i^2 \xi_i\otimes \xi_i
    +(\frac{1}{2}\nabla  {v}_{i-1}\otimes\nabla  {v}_{i-1}+\sym\nabla   {w}_{i-1})
    -(\frac{1}{2}\nabla  {v}_i\otimes\nabla  {v}_i+\sym\nabla   {w}_i)\\
    & =\frac{1}{ {\mu}_i}\Gamma_1\nabla^2  {v}_{i-1}
    -\frac{1}{ {\mu}_i}(\partial_{s}\Gamma_2+(\partial_{s}\Gamma_1)(\partial_{t}\Gamma_1))\sym(\nabla {a}_i\otimes \xi_i)\\
    & -\frac{1}{2 {\mu}_i^2}|\partial_{s}\Gamma_1|^2\nabla {a}_i\otimes\nabla {a}_i.
\end{align*}
Then by \eqref{881} \eqref{882}
\begin{equation}
    D_{q+1}=A-\tilde A+\sum_{i=1}^{N_*}\E_i.
\end{equation}
Recalling \eqref{bounds}, one obtains
\begin{equation}
\begin{aligned}
    \|\E_i\|_0
    &\lesssim\frac{1}{\mu_i}\|\Gamma_1\|_0\|\nabla^2 v_i\|_0 +\frac{1}{\mu_i}(\|\partial_s\Gamma_2\|_0+\|\partial_s\Gamma_1\|_0\|\partial_t\Gamma_1\|_0)\|\nabla a_i\|_0\\
    &+\frac{1}{\mu_i^2}\|\partial_s\Gamma_1\|_0^2\|\nabla a_i\|_0^2\\
    &\lesssim\frac{1}{\mu_i}\delta_{q+1}^\frac{1}{2}\cdot \delta_{q+1}^\frac{1}{2}\mu_{i-1}
    + \frac{1}{\mu_i}\cdot\delta_{q+1}^\frac{1}{2}\cdot\sigma\delta_{q+1}^\frac{1}{2}l^{-1}
    +\frac{1}{\mu_i^2}(\sigma\delta_{q+1}^\frac{1}{2}l^{-1})^2\\
    &\lesssim \delta_{q+1}\frac{\mu_{i-1}}{\mu_i}
\end{aligned}
\end{equation}
Besides, since $A\in C^{2,\kappa}(\bar\Omega,\S)$,
\begin{equation*}
    \|A-\tilde A\|_0\lesssim \|A\|_1 l\lesssim l.
\end{equation*}
So we conclude
\begin{equation*}
    \|D_{q+1}\|_0\lesssim l+\delta_{q+1}(\frac{\mu_0}{\mu_1}
    +\frac{\mu_1}{\mu_2}+...+\frac{\mu_{N_*-1}}{\mu_{N_*}}).
\end{equation*}
We now take frequencies $\mu_i$ as
\begin{equation*}
    \mu_i:=\mu_0^{1-\frac{i}{N_*}}\mu_{N_*}^\frac{i}{N_*}
\end{equation*}
and require
\begin{equation}\label{ineq**}
    l\le \frac{\sigma\delta_{q+2}}{C},\quad 
    \frac{\mu_0}{\mu_{N_*}}\le
    (\frac{\sigma\delta_{q+2}}{C\delta_{q+1}})^{N_*}
\end{equation}
for some universal constant $C>1$ large enough. 
In particular \eqref{ineq03} holds.
Now we get
\begin{equation*}
\|D_{q+1}\|_0\lesssim l+\delta_{q+1}(\frac{\mu_0}{\mu_{N_*}})^\frac{1}{N_*}\le{\sigma}\delta_{q+2}
\end{equation*}
which is exactly \eqref{prop1.2}. This concludes the proof of Proposition \ref{prop1}.
\end{proof}

\textbf{Step 4 (Conclusion): $(V_q,W_q)\to(v,w)$.}
Now it is easy to prove Theorem \ref{thm1}. First take $\sigma=\sigma_*/3$ and take $K$ large enough, such that \eqref{K3} \eqref{014} \eqref{K6} \eqref{K2} hold.

Next we determine the parameters $b,c$ depending on $N_*,\alpha$, and $a>1$ depending on $K,\sigma,\epsilon$ and $b,c$. The inequalities we require are \eqref{ineq*} \eqref{ineq**} (recall \eqref{mu0} \eqref{l})
\begin{gather*}
    \delta_1^{-N_*}\le \delta_0^\frac{1}{2}\lambda_0,\\
    \frac{\sigma\delta_{q+1}^\frac{1}{2}}{K\delta_q^\frac{1}{2}\lambda_q}\le\frac{\sigma\delta_{q+2}}{C},\\
    K\frac{\delta_q^\frac{1}{2}\lambda_q}{\delta_{q+1}^\frac{1}{2}\lambda_{q+1}}\le
    (\frac{\sigma\delta_{q+2}}{C\delta_{q+1}})^{N_*}
\end{gather*}
and \eqref{ineq01} \eqref{ineq02} (recall $\sigma=\sigma_*/3$)
\begin{equation}\label{907}
    \delta_1\le{\sigma},\quad
    \delta_{q+2}\le {\sigma}\delta_{q+1}.
\end{equation}
Recall $\delta_q=a^{-b^q},\lambda_q=a^{cb^q}$. By direct computation, the first three are equivalent to
\begin{equation}\label{ineq901}
    c\ge N_* b+\frac{1}{2}+\frac{1}{(b-1)b^q}\log_a(\frac{CK}{\sigma^{N_*}})
\end{equation}
On the other hand, by \eqref{prop1.0} \eqref{prop1.1} and interpolation
\begin{equation}\label{interp}
    \|V_{q+1}-V_q\|_{1+\alpha}\leq \|V_{q+1}-V_q\|_1^{1-\alpha} \|V_{q+1}-V_q\|_2^{\alpha}\le K\delta_{q+1}^\frac{1}{2}\lambda_{q+1}^\alpha.
\end{equation}
So convergence in the $C^{1,\alpha}$ norm is equivalent to
\begin{equation}\label{ineq902}
    -\frac{1}{2}+c\alpha<0.
\end{equation}
As long as
\begin{equation*}
    \alpha<\frac{1}{1+2N_*}=\frac{1}{1+n+n^2},
\end{equation*}
\eqref{907} \eqref{ineq901} and \eqref{ineq902} can be achieved by taking $b>1$ close to $1$, $c>N_*+\frac{1}{2}$ close to $N_*+\frac{1}{2}$ and $a>1$ sufficiently large.
We fix such $b,c$ and below take $a$ even larger.

\eqref{011} \eqref{012} \eqref{014} in Step 2 yield the induction assumptions \eqref{prop1.a1} \eqref{prop1.a2} and \eqref{prop1.a3} of Proposition \ref{prop1} at $q=0$. Moreover, as
\begin{equation*}
    \|(V_0,W_0)\|_1\le \frac{\sqrt{K}}{2},\quad
    \|(V_{q+1}-V_q, W_{q+1}-W_q)\|_1\le K\delta_{q+1}^\frac{1}{2},
\end{equation*}
by taking $a>1$ large enough, we may assume
\begin{equation*}
    \|(V_q,W_q)\|_1\le \sqrt{K},\ \forall\ q\ge 0.
\end{equation*}
Now successive use of Proposition \ref{prop1} produces a series of $(V_q,W_q)$ satisfying \eqref{prop1.0}-\eqref{prop1.2} and hence the interpolation bound \eqref{interp}, so
\begin{align*}
    &(V_q,W_q)\to (v,w)\ \text{in}\ C^{1,\alpha}(\bar\Omega)\times C^{1,\alpha}(\bar\Omega,\R^n),\\
    &D_q\to 0\ \text{in}\ C^0(\bar\Omega,\S).
\end{align*}
Recall
\begin{equation*}
    D_q=A-\frac{1}{2}\nabla V_q\otimes\nabla V_q-\sym\nabla W_q-\delta_{q+1}Id.
\end{equation*}
Taking limit, we get
\begin{equation*}
    A-\frac{1}{2}\nabla v\otimes\nabla v-\sym\nabla w=0.
\end{equation*}
Moreover by \eqref{1002} \eqref{013} and \eqref{prop1.0},
\begin{align*}
    \|v-v^b\|_0
    &\le \|\bar v_N-v^b\|_0+\|V_0-\bar v_N\|_0+\sum_{q=0}^{+\infty}\|V_{q+1}-V_q\|_0\\
    &\le \frac{\epsilon}{3}+\sigma\delta_1+\sum_{q=0}^{+\infty}K\delta_{q+1}^\frac{1}{2}\\
    &\le \epsilon
\end{align*}
and similarly $\|w-w^b\|_0\le\epsilon$, as long as $a$ is large enough. So \eqref{thm1.1} and \eqref{thm1.2} are justified. This concludes the proof of Theorem \ref{thm1}.

\section{Cut-off technique.}\label{sec2}
In this section we take into consideration the boundary value $v=g$. As in \cite{cao23dirichlet}, we adapt the arguments in the last section by further introducing a cut-off procedure. More precisely, we will cut-off at the level set of a function $\psi$, which represents the size of $D^b$, and in each stage $(V_q,W_q)\to(V_{q+1},W_{q+1})$, we only change the data in the set
\begin{equation*}
    \{\psi>\delta_{q+2}\}\Subset\Omega.
\end{equation*}

Let's now explain the details. First we take $v^b\in C^{2,\kappa}(\bar\Omega)$ to solve
\begin{equation}\label{gggg}
    \Delta v^b = 0\ \text{in}\ \Omega,\quad
    v^b = g\ \text{on}\ \partial\Omega.
\end{equation}
With this initial data $v^b$, we will show that 
\begin{thm}\label{mainthm3}
    Under the conditions of Theorem \ref{mainthm2} or \ref{mainthm}, for any $\epsilon>0$, there exists a very weak solution $v\in C^{1,\alpha}(\bar\Omega)$ to \eqref{eq1} such that
\begin{equation*}
    \|v-v^b\|_0\le \epsilon,\quad v=v^b=g\ \text{ on }\ \partial\Omega.
\end{equation*}
\end{thm}
This clearly yields Theorem \ref{mainthm2} and \ref{mainthm}. The proof is in the same line with Section \ref{sec1} and essentially follows \cite{cao23dirichlet}. Remark \ref{rem995} explains in more detail the construction of $\tilde D$, corresponding to the quantity $H_{q+1}$ in \cite[page 19]{cao23dirichlet}.

\textbf{Part one: solve \eqref{2}.} 
Assuming that first $f\in C^{0,\kappa}(\bar\Omega)$ and $f>0$ aligns with the assumption in Theorem \ref{mainthm}. Given $g\in C^{2,\kappa}(\partial\Omega)$, we let $v^b$ solves \eqref{gggg} and take $\psi$ such that
\begin{align*}
    -(2n-2)\Delta \psi &= f-\sigma_2(\nabla^2 v^b) \quad \text{in}\ \Omega,\\
    \psi &= 0 \quad\text{on}\ \partial\Omega.
\end{align*}
Then by Newton inequality  $\sigma_2(\nabla^2 v^b)\leq \frac{n-1}{2n} (\Delta v^b)^2  = 0$, so
\begin{equation*}
    f-\sigma_2(\nabla^2 v^b)>0.
\end{equation*}
For general $f\in C^{0,\kappa}(\bar\Omega)$ as in Theorem \ref{mainthm2}, we pick $v^b\in C^{3}(\bar\Omega)$ such that 
\begin{equation*}
f-\sigma_2(\nabla^2 v^b)>0,
\end{equation*}
e.g. $v^b(x)=C(x_1^2-x_2^2)$ for $C>1$ sufficiently large. Then we take $g=v^b\mid_{\partial\Omega}$ and $\psi$ as above. 
In both cases, $\psi\in C^{2,\kappa}(\bar\Omega)$ and by strong maximum principle
\begin{equation*}
    \psi>0\ \text{in}\ \Omega,\quad \psi=0\ \text{on}\ \partial\Omega.
\end{equation*}
Let
\begin{equation*}
A=\psi Id+\frac{1}{2}\nabla v^b\otimes\nabla v^b, 
\end{equation*}
Then
\begin{equation*}
    -L(A)=-(2n-2)\Delta \psi+\sigma_2(\nabla^2 v^b)=f.
\end{equation*}

\textbf{Part two: solve \eqref{4}.} Again, we will construct approximate solutions $(V_q,W_q)$ inductively. 

\textbf{Step 1 (Preparation).}
Set $w^b=0$. Define the initial deficit matrix by 
\begin{equation}\label{z0001}
    D^b:=A-\frac{1}{2}\nabla v^b\otimes\nabla v^b-\sym\nabla w^b=\psi Id.
\end{equation}
Unlike before, here we do not need to construct $(\bar v_N,\bar w_N)$ to reduce $D^b$, since $D^b$ is merely a multiple of $Id$.  In particular, Lemma \ref{lem1} is never used. The fact that $D^b$ is merely a (non-negative) multiple of $Id$ is crucial to the cut-off technique: we will cut-off at the level set of $\psi$.

Let's introduce some notations. We decompose
\begin{equation*}
    Id=\sum_{i=1}^{N_*}d_i^{*2}\xi_i\otimes\xi_i.
\end{equation*}
for some constants $d_i^*\in[c_*,C_*]$, by applying Lemma \ref{lem2} to $Id$. 

Let
\begin{align*}
    &\Omega_q:=\{x\in\Omega:\psi(x)>2\delta_{q}\},\\
    &\tilde\Omega_q:=\{x\in\Omega:\psi(x)>\frac{3}{2}\delta_{q}\}.
\end{align*}
Clearly $\Omega_q\Subset\tilde\Omega_q\Subset\Omega_{q+1}\Subset\Omega$ and
\begin{equation*}
    dist(\tilde\Omega_q,\partial\Omega)\ge dist(\Omega_q,\partial\tilde\Omega_q)\gtrsim \delta_{q}.
\end{equation*}
Here and below the implicit constants may depend on $\|\psi\|_1$.

Let $\eta_q$ be a smooth cut-off function subordinate to $\Omega_q\subset\tilde\Omega_q$, i.e.
\begin{align*}
    &\eta_q=1\quad\text{in}\ \Omega_q,\\
    &\eta_q=0\quad\text{in}\ \Omega-\tilde\Omega_q,\\
    &\|\nabla^k \eta_q\|_0\lesssim \delta_{q}^{-k},\ \forall 0\le k\le 3.
\end{align*}
Let $\psi_q$ be a smooth truncation of $\psi$ at $\delta_q$, defined by
\begin{equation}
\psi_{q}:=\eta_q^2\delta_{q}+(1-\eta_q^2)\psi.
\end{equation}
Note that
\begin{equation*}
    \|\nabla\psi_q\|_0\lesssim \delta_q \|\nabla\eta_q\|_0 + \|\nabla\eta_q\|_0\|\psi\|_{C^0(\Omega-\Omega_q)} +\|\nabla\psi\|_0\lesssim 1.
\end{equation*}
\textbf{Step 2 ($0$th approximate): $(v^b,w^b)\to(V_0,W_0)$.}
By mollification in $\tilde\Omega_0$, we may assume $v^b\in C^\infty(\tilde\Omega_0)$. Let
\begin{equation*}
    \hat a_i:=\eta_1(\psi-\delta_1)^\frac{1}{2}d_i^*,
\end{equation*}
in other words (recall $D^b=\psi Id$)
\begin{equation}\label{z0002}
    \eta_1^2(D^b-\delta_1 Id)=\sum_{i=1}^{N_*}\hat a_i^2\xi_i\otimes\xi_i.
\end{equation}
Then
\begin{gather}
    \|\nabla^k\hat a_i\|_0\lesssim \delta_1^{-k},\ \forall0\le k\le 2, \label{996}\\
    \hat{a}_i = 0 \quad\text{in }\Omega-\tilde\Omega_1. \label{2020}
\end{gather}

As in Step 2 in Section \ref{sec1}, we define
\begin{align}
    &( \hat v_0, \hat w_0)=(v^b, w^b),\\
    & \hat{v}_i= \hat{v}_{i-1}+\frac{1}{\hat{\mu}_i}\Gamma_1(\hat{a}_i, \hat{\mu}_i x\cdot \xi_i),\\
    & \hat{w}_i= \hat{w}_{i-1}-\frac{1}{\hat{\mu}_i}\Gamma_1(\hat{a}_i,\hat{\mu}_i x\cdot \xi_i)\nabla \hat{v}_{i-1}+\frac{1}{\hat{\mu}_i}\Gamma_2(\hat{a}_i,\hat{\mu}_i x\cdot \xi_i) \xi_i,\\
    &(V_0,W_0)=(\hat v_{N_*},\hat w_{N_*}),
\end{align}
with constants $1\ll\hat{\mu}_1<...<\hat{\mu}_{N_*}$ determined later. Denote the $0$th stage deficit matrix by
\begin{equation}\label{2028}
     D_0:=A-\frac{1}{2}\nabla  V_0\otimes\nabla V_0-\sym\nabla W_0-\psi_1 Id.
\end{equation}
Note that by \eqref{2020}
\begin{equation}\label{hhhh}
(V_0,W_0,D_0)=(v^b,w^b,0)\quad\text{in}\ \Omega-\tilde\Omega_1.
\end{equation}
The $i$-th step error is
\begin{equation*}
\begin{aligned}
    \hat{\E}_i
    &=\hat{a}_i^2 \xi_i\otimes \xi_i
    +(\frac{1}{2}\nabla  \hat{v}_{i-1}\otimes\nabla  \hat{v}_{i-1}+\sym\nabla   \hat{w}_{i-1})
    -(\frac{1}{2}\nabla  \hat{v}_i\otimes\nabla  \hat{v}_i+\sym\nabla   \hat{w}_i)\\
    & =\frac{1}{ \hat{\mu}_i}\Gamma_1\nabla^2  \hat{v}_{i-1}
    -\frac{1}{ \hat{\mu}_i}[\partial_{s}\Gamma_2+(\partial_{s}\Gamma_1)(\partial_{t}\Gamma_1)]\sym(\nabla \hat{a}_i\otimes \xi_i)\\
    & -\frac{1}{2\hat{\mu}_i^2}|\partial_{s}\Gamma_1|^2\nabla \hat{a}_i\otimes\nabla \hat{a}_i
\end{aligned}
\end{equation*}
Injecting \eqref{z0001} and \eqref{z0002} into \eqref{2028}
\begin{align*}
     D_0&=A-\frac{1}{2}\nabla  V_0\otimes\nabla V_0-\sym\nabla W_0-\eta_1^2\delta_1 Id-(1-\eta_1^2)\psi Id\\
     &=A-D^b-\frac{1}{2}\nabla  V_0\otimes\nabla V_0-\sym\nabla W_0+\eta_1^2(D^b-\delta_1 Id)\\
     &=\frac{1}{2}\nabla  v^b\otimes\nabla v^b+\sym\nabla w^b
     -\frac{1}{2}\nabla V_0\otimes\nabla V_0-\sym\nabla W_0
     +\sum_{i=1}^{N_*}\hat a_i^2\xi_i\otimes\xi_i\\
     &=\sum_{i=1}^{N_*}\hat\E_i
\end{align*}
We turn to required estimates. Using \eqref{bounds} \eqref{996}
\begin{align*}
    \|\nabla^2\hat v_i-\nabla^2\hat v_{i-1}\|_0
    &\lesssim\frac{1}{\hat\mu_i}(\|\nabla^2\hat a_i\|_0+\hat\mu_i\|\nabla\hat a_i\|_0+\hat\mu_i^2\|\hat a_i\|_0)\\
    &\lesssim\frac{1}{\hat \mu_i}(\delta_1^{-2}+\hat\mu_i\delta_1^{-1}+\hat\mu_i^2)\\
    &\lesssim \hat\mu_i
\end{align*}
where we assumed
\begin{equation}\label{a1}
    \hat\mu_i\ge \delta_1^{-1}.
\end{equation}
Using this and \eqref{bounds} \eqref{996}, we get
\begin{align*}\label{bd0002}
    \|\hat\E_i\|_0
    \lesssim \frac{1}{\hat\mu_i}\|\hat a_i\|_0\|\nabla^2\hat v_{i-1}\|_0
    +\frac{1}{\hat\mu_i}\|\hat a_i\|_0\|\nabla\hat a_i\|_0
    +\frac{1}{\mu_i^2}\|\nabla\hat a_i\|_0^2
    \lesssim \frac{\hat\mu_{i-1}}{\hat\mu_i}
\end{align*}
where we write $\hat\mu_0=1$. This is the same as in Step 2 of Section \ref{sec1}. Similarly we deduce
\begin{gather*}
    \| D_0\|_0\le\sum_{i=1}^{N_*}\|\hat\E_i\|_0\lesssim \frac{1}{\hat{\mu}_1}+\frac{\hat{\mu}_1}{\hat{\mu}_2}+...+\frac{\hat{\mu}_{N_*-1}}{\hat{\mu}_{N_*}},\\
    \|(V_0,W_0)\|_2\lesssim \hat{\mu}_{N_*}.
\end{gather*}
Hence we take
\begin{equation*}
    \hat{\mu}_i:=(\frac{C}{\sigma\delta_1})^i
\end{equation*}
in particular \eqref{a1} holds. Then, assuming
\begin{equation*}
    \delta_1^{-N_*}\le \delta_0^\frac{1}{2}\lambda_0
\end{equation*}
and $K>1$ large enough, we conclude
\begin{gather*}
    \| D_0\|_0\le {\sigma}\delta_1,\quad
    \|(V_0,W_0)\|_2\le K\delta_0^\frac{1}{2}\lambda_0,\\
    \|(V_0-v^b,W_0-w^b)\|_0\le\sigma\delta_1,\quad
    \|(V_0,W_0)\|_1\le \frac{\sqrt{K}}{2}.
\end{gather*}
Together with \eqref{hhhh}, these ensure the assumptions \eqref{prop2.a1}-\eqref{prop2.a4} below at $q=0$.

\textbf{Step 3 (Induction): $(V_q,W_q)\to(V_{q+1},W_{q+1})$.} 
Denote the $q$-th stage deficit matrix by
\begin{equation}\label{ccc}
   D_q:= A-\frac{1}{2}\nabla V_q\otimes\nabla V_q-\sym\nabla W_q-\psi_{q+1}Id.
\end{equation}
Similar to Proposition \ref{prop1}, we will prove
\begin{prop}\label{prop2}
    There exist some positive universal constants $\sigma\in(0,\sigma_*/3)$ and $K>1$, such that: if $(V_q,W_q)\in C^2(\bar\Omega)\times C^2(\bar\Omega,\R^n)$ satisfy
    \begin{gather}
         \|(V_q,W_q)\|_1\le \sqrt K, \label{prop2.a1}\\
         \|(V_q,W_q)\|_2\le K\delta_q^\frac{1}{2}\lambda_q, \label{prop2.a2}\\
         \|D_q\|_0\le{\sigma}\delta_{q+1}.\label{prop2.a3}
    \end{gather}
    and
    \begin{equation}\label{prop2.a4}
        (V_q,W_q,D_q)=(v^b,w^b,0)\quad\text{in}\ \Omega-\tilde\Omega_{q+1},
    \end{equation}
    we can construct $(V_{q+1},W_{q+1})\in C^2(\bar\Omega)\times C^2(\bar\Omega,\R^n)$ with
    \begin{gather}
        \|(V_{q+1}-V_q, W_{q+1}-W_q)\|_1\le K\delta_{q+1}^\frac{1}{2}\label{prop2.0}\\
        \|(V_{q+1},W_{q+1})\|_2\le K\delta_{q+1}^\frac{1}{2}\lambda_{q+1},\label{prop2.1}\\
         \|D_{q+1}\|_0\le{\sigma}\delta_{q+2}.\label{prop2.2}
    \end{gather}
    and
    \begin{equation}\label{prop2.4}
        (V_{q+1},W_{q+1},D_{q+1})=(v^b,w^b,0)\quad\text{in}\ \Omega-\tilde\Omega_{q+2},\\
    \end{equation}
\end{prop}


\begin{proof}
Let $\tilde A,\tilde v,\tilde w,\tilde\psi_{q+1}$ be the mollification of $A,V_q,W_q,\psi_{q+1}$, respectively, at length $l$ in $\tilde\Omega_{q+2}$, where as in \eqref{mu0} \eqref{l}
\begin{equation}\label{ll}
     \mu_0 := K\delta_{q+1}^{-\frac{1}{2}}\delta_{q}^\frac{1}{2}\lambda_q,\quad l:=\frac{\sigma}{C\mu_0}.
\end{equation}
In particular \eqref{prop2.a2} is
\begin{equation}\label{prop2.a2'}
    \|(V_q,W_q)\|_2\le \delta_{q+1}^\frac{1}{2}\mu_0.
\end{equation}
Let $\eta$ be a smooth cut-off function such that
\begin{gather*}
    \eta=1\quad\text{if } \psi\ge\frac{5}{4}\delta_{q+1},\\
    \eta=0\quad\text{if } \psi\le\delta_{q+1},
\end{gather*}
and
\begin{equation}
    \|\nabla^k\eta\|_0\lesssim \delta_{q+1}^{-k},\ \forall 0\le k\le 3.
\end{equation}
We assume
\begin{equation}\label{ineq999}
    l\le \frac{\delta_{q+2}}{4\|\psi\|_1+1}.
\end{equation}
Then
\begin{enumerate}
    \item[($l_1$)] $l\le dist(\tilde\Omega_{q+2},\partial\Omega)$ and the mollification is well-defined;
    \item[($l_2$)] $\eta=1$ in the $l$-neighborhood of $\tilde\Omega_{q+1}$.
\end{enumerate} 
Define
\begin{equation}\label{tildeD}
\begin{aligned}
    \tilde D &:= \eta^2(\tilde A-\frac{1}{2} \nabla\tilde v\otimes\nabla\tilde v-\sym\nabla\tilde w-\delta_{q+2}Id)\\
    &\quad+(1-\eta^2)\eta_{q+2}^2(\tilde\psi_{q+1}-\delta_{q+2})Id. 
\end{aligned}
\end{equation}
Compared with \eqref{ccc}, the first line is
\begin{equation*}
\begin{aligned}
    \eta^2(\tilde A-\frac{1}{2} \nabla\tilde v\otimes\nabla\tilde v-\sym\nabla\tilde w-\delta_{q+2}Id) &=\eta^2 D_q*\phi_l+\eta^2(\tilde\psi_{q+1}-\delta_{q+2})Id.
\end{aligned}
\end{equation*}
Since $\eta^2\eta_{q+2}^2=\eta^2$, the second line is
\begin{equation*}
    (1-\eta^2)\eta_{q+2}^2(\tilde\psi_{q+1}-\delta_{q+2})Id = \eta_{q+2}^2(\tilde\psi_{q+1}-\delta_{q+2})Id - \eta^2(\tilde\psi_{q+1}-\delta_{q+2})Id.
\end{equation*}
Therefore
\begin{equation}
\begin{aligned}
    \tilde D &= \eta_{q+2}^2(\tilde\psi_{q+1} -\delta_{q+2})Id\\ &\quad +\eta^2 D_q*\phi_l +\frac{1}{2}\eta^2[(\nabla V_q\otimes\nabla V_q)*\phi_l-\nabla\tilde v\otimes\nabla\tilde v].
\end{aligned}
\end{equation}
Now we construct $a_i$ in two cases separately.

Case 1: $\psi\ge \delta_{q+1}$. Then by \eqref{phil3} \eqref{prop2.a3}
\begin{equation*}
    \|\eta^2 D_q*\phi_l\|_0 \le \|D_q\|_0\le \sigma\delta_{q+1},
\end{equation*}
and by \eqref{ll} \eqref{prop2.a2'} and \eqref{phil2}
\begin{equation*}
    \|\frac{1}{2}\eta^2[(\nabla V_q\otimes\nabla V_q)*\phi_l-\nabla\tilde v\otimes\nabla\tilde v]\|_0 \lesssim \|V_q\|_2^2 l^2 \le \sigma\delta_{q+1},
\end{equation*}
so
\begin{equation*}
   \| \tilde D -\eta_{q+2}^2(\tilde\psi_{q+1} -\delta_{q+2})Id\|_0 \le 2\sigma\delta_{q+1}.
\end{equation*}
Let us assume $\sigma\le \frac{\sigma_*}{3}\le \frac{1}{6}$, such that
\begin{equation} \label{deltasigma}
    \delta_{q+2}\le \sigma\delta_{q+1}\le \frac{1}{6}\delta_{q+1}.
\end{equation}
So from $\psi\ge \delta_{q+1}$, we have $\eta_{q+2}=1$. And by \eqref{ineq999} and the definition of 
\begin{equation*}
 \psi_{q+1}=\eta^2_{q+1} \delta_{q+1} +(1-\eta^2_{q+1}) \psi,   
\end{equation*}
we obtain
\begin{equation*}
    \eta_{q+2}^2(\tilde\psi_{q+1}-\delta_{q+2}) \ge \delta_{q+1}-\frac{5}{4}\delta_{q+2}\ge \frac{2}{3}\delta_{q+1}.
\end{equation*}
Therefore
\begin{equation*}
    \|\frac{\tilde D}{\eta_{q+2}^2(\tilde\psi_{q+1}-\delta_{q+2})}-Id\|_0\le 3\sigma\le\sigma_*.
\end{equation*}
Note that since
\begin{equation*}
    \|\frac{\tilde D}{\eta_{q+2}^2(\tilde\psi_{q+1}-\delta_{q+2})}-Id\|_k\lesssim l^{-k},\forall 0\le k\le 3
\end{equation*}
where we omit the universal constant $\sigma_\ast$.
Applying Lemma \ref{lem2}, we get
\begin{equation}
    \tilde D=\sum_{i=1}^{N_*} a_{i,1}^2\xi_i\otimes\xi_i,
\end{equation}
and
\begin{equation*}
    \|\nabla^k a_{i,1}\|_0
    \lesssim \delta_{q+1}^\frac{1}{2}l^{-k},\ \forall 0\le k\le 3.
\end{equation*}

Case 2: $\psi\le\delta_{q+1}$. Then $\eta=0$ and
\begin{equation*}
    \tilde D = \eta_{q+2}^2(\tilde\psi_{q+1} -\delta_{q+2})Id.
\end{equation*}
Recall $Id = \sum_{i=1}^{N_*}d_i^{*2}\xi_i\otimes\xi_i$. By \eqref{ineq999} and the definition of $\tilde\Omega_{q+2}$
\begin{equation}
    \tilde\psi_{q+1}-\delta_{q+2}\ge \frac{3}{2}\delta_{q+2}-\frac{1}{4}\delta_{q+2}-\delta_{q+2}=\frac{1}{4} \delta_{q+2}\quad\text{in}\ \tilde\Omega_{q+2}.
\end{equation}
Let
\begin{equation*}
    a_{i,2}= \eta_{q+2}(\tilde\psi_{q+1} -\delta_{q+2})^{\frac{1}{2}}d_i^*
\end{equation*}
Note that $a_{i,2}=0$ in $\Omega-\tilde\Omega_{q+2}$. Then
\begin{equation*}
    \tilde D=\sum_{i=1}^{N_*} a_{i,2}^2\xi_i\otimes\xi_i,
\end{equation*}
and
\begin{equation*}
    \|\nabla^k a_{i,2}\|_0 \lesssim \delta_{q+1}^\frac{1}{2}l^{-k},\ \forall 0\le k\le 3.
\end{equation*}

Now we define $a_i$ as
\begin{equation}
    a_i = \begin{cases}
        a_{i,1}, & \text{if }\psi\ge \delta_{q+1},\\
        a_{i,2}, &\text{if }\psi\le\delta_{q+1}.
    \end{cases}
\end{equation}
Then
\begin{equation}
    \tilde D=\sum_{i=1}^{N_*} a_i^2\xi_i\otimes\xi_i.
\end{equation}
Note that $a_{i,1}$ and $a_{i,2}$ are both obtained by the decomposition of $\tilde D$, and $\tilde D$ is smooth, so $a_{i,1}$ and $a_{i,2}$ patches smoothly. Thus $a_i$ is smooth with
\begin{equation*}
    \|\nabla^k a_i\|_0
    \lesssim \delta_{q+1}^\frac{1}{2}l^{-k},\ \forall 0\le k\le 3
\end{equation*}
and
\begin{equation}
    a_i=0\quad\text{in}\ \Omega-\tilde\Omega_{q+2}. \label{a-i}
\end{equation}
\begin{rem}\label{rem995}
    The above arguments aim to find such $a_i$ smoothly compactly supported in $\tilde\Omega_{q+2}$, which are used as usual to reduce $D_q$ to near $\delta_{q+2}Id$. To this end, we have expected to construct $a_i$ in two domains separately and then glue them together. In $\tilde\Omega_{q+1}$ we resort to Lemma \ref{lem2}. While in $\Omega-\Omega_{q+1}$ we have $D_q=\psi_{q+1}Id$, so we simply  proceed as in Step 2. But this gluing procedure brings extra error which should be cancelled by $a_i$. So instead we first glue the deficit matrix to obtain $\tilde D$, so that $a_i$ automatically glue together in a smooth way. The exact expression of $\tilde D$ is determined by \eqref{eq*}.
\end{rem}

Let
\begin{align}
    &(v_0, w_0)=\eta^2(\tilde v,\tilde w)+(1-\eta^2)(V_q,W_q),\\
    & v_i= v_{i-1}+\frac{1}{\mu_i}\Gamma_1(a_i, \mu_i x\cdot \xi_i),\\
    & w_i= w_{i-1}-\frac{1}{\mu_i}\Gamma_1(a_i,\mu_i x\cdot \xi_i)\nabla v_{i-1}+\frac{1}{\mu_i}\Gamma_2(a_i,\mu_i x\cdot \xi_i) \xi_i,\\
    & (V_{q+1},W_{q+1})=(v_{N_*},w_{N_*}).
\end{align}
From the definition of $\eta$ and \eqref{deltasigma}  \eqref{a-i}, it is clear that \eqref{prop2.4} holds, i.e.
\begin{equation*}
    (V_{q+1},W_{q+1},D_{q+1})=(V_q,W_q,D_q)=(v^b,w^b,0)\quad\text{in}\ \Omega-\tilde\Omega_{q+2}.
\end{equation*}
We derive necessary estimates. We have
\begin{equation*}
    (v_0,w_0)-(V_q,W_q)=\eta^2(\tilde v-V_q,\tilde w-W_q).
\end{equation*}
Using \eqref{phil1} \eqref{ll} \eqref{prop2.a2'} \eqref{ineq999} \eqref{deltasigma} and $\|\nabla\eta\|_0\lesssim \delta_{q+1}^{-1}$, we get
\begin{align*}
    \|(v_0,w_0)-(V_q,W_q)\|_1
    &\lesssim \|\eta^2\|_0\|(\tilde v-V_q,\tilde w-W_q)\|_1+\|\nabla\eta^2\|_0\|(\tilde v-V_q,\tilde w-W_q)\|_0\\
    &\lesssim\|(V_q,W_q)\|_2 (l+l^2\|\nabla\eta\|_0)\\
    &\lesssim \sigma\delta_{q+1}^\frac{1}{2}
\end{align*}
and
\begin{equation*}
    \|(v_0,w_0)-(V_q,W_q)\|_{2+k}\lesssim \|(V_q,W_q)\|_2 l^{-k}\lesssim\delta_{q+1}^\frac{1}{2}\mu_0 l^{-k},\  k= 0,1.
\end{equation*}
As usual, the $i$-th step error is
\begin{equation}
    \E_i:={a}_i^2 \xi_i\otimes \xi_i
    +(\frac{1}{2}\nabla  {v}_{i-1}\otimes\nabla  {v}_{i-1}+\sym\nabla   {w}_{i-1})
    -(\frac{1}{2}\nabla  {v}_i\otimes\nabla  {v}_i+\sym\nabla   {w}_i).
\end{equation}
Using the above estimates, we can proceed as in Step 3 of Section 2 to get \eqref{prop2.0} \eqref{prop2.1} and
\begin{equation}\label{dddd}
    \|\E_i\|_0\lesssim \delta_{q+1}\frac{\mu_{i-1}}{\mu_i}.
\end{equation}
It remains to derive \eqref{prop2.2}. By the definition of $D_{q+1}$
\begin{align}
    D_{q+1} &= A-\frac{1}{2}\nabla V_{q+1}\otimes\nabla V_{q+1}-\sym\nabla W_{q+1}-\psi_{q+2}Id \notag\\
    &= A-\frac{1}{2}\nabla v_0\otimes\nabla v_0-\sym\nabla w_0-\psi_{q+2}Id \label{ttt}\\
    &\quad+\sum_{i=1}^{N_*}(\frac{1}{2}\nabla  {v}_{i-1}\otimes\nabla  {v}_{i-1}+\sym\nabla   {w}_{i-1})
    -(\frac{1}{2}\nabla  {v}_i\otimes\nabla  {v}_i+\sym\nabla   {w}_i). \label{sss}
\end{align}
We reformulate it. First, by the construction
\begin{align*}
    \nabla v_0 &= \eta^2\nabla\tilde v +(1-\eta^2)\nabla V_q +(\nabla\eta^2)(\tilde v-V_q),\\
    \nabla w_0 &= \eta^2\nabla\tilde w +(1-\eta^2)\nabla W_q +(\nabla\eta^2)\otimes(\tilde w-W_q),
\end{align*}
so
\begin{align*}
    &\nabla v_0\otimes \nabla v_0\\
    &= \eta^4\nabla \tilde v\otimes\nabla\tilde v +(1-\eta^2)^2\nabla V_q\otimes\nabla V_q +2\eta^2(1-\eta^2)\sym\nabla\tilde v\otimes\nabla V_q\\
    &\quad +2\sym[(\nabla \eta^2)(\tilde v-V_q)]\otimes[(\eta^2\nabla\tilde v+(1-\eta^2)\nabla V_q]\\
    &\quad +[(\nabla \eta^2)(\tilde v-V_q)]\otimes[(\nabla \eta^2)(\tilde v-V_q)]\\
    &= \eta^2\nabla\tilde v\otimes\nabla\tilde v +(1-\eta^2)\nabla V_q\otimes\nabla V_q -\eta^2(1-\eta^2)\nabla(\tilde v-V_q)\otimes\nabla(\tilde v-V_q)\\
    &\quad +2\sym[(\nabla \eta^2)(\tilde v-V_q)]\otimes[(\eta^2\nabla\tilde v+(1-\eta^2)\nabla V_q]\\
    &\quad +[(\nabla \eta^2)(\tilde v-V_q)]\otimes[(\nabla \eta^2)(\tilde v-V_q)].
\end{align*}
Since $V_q=0$ in $\Omega-\tilde\Omega_{q+1}$ and $\eta=1$ in the $l$-neighbourhood of $\tilde\Omega_{q+1}$, there holds
\begin{equation*}
    \eta^2(1-\eta^2)\nabla(\tilde v-V_q)\otimes\nabla(\tilde v-V_q) =0.
\end{equation*}
Next, by the definition of $\psi_q,\eta,\eta_q$, we find
\begin{eqnarray*}
    \psi_{q+2}  &=&  \eta^2_{q+2} \delta_{q+2} +(1-\eta^2_{q+2}) \psi\\
    &=& \eta^2_{q+2} \delta_{q+2} +(1-\eta^2_{q+1})(1-\eta^2_{q+2}) \psi \\
    &=& \eta^2_{q+2} \delta_{q+2} + (1-\eta^2_{q+2})[ (1-\eta^2_{q+1})\psi +\eta^2_{q+1}\delta_{q+1}-\eta^2_{q+1}\delta_{q+1}]\\
    &=& \eta^2_{q+2} \delta_{q+2} + (1-\eta^2_{q+2})[ \psi_{q+1}-\eta^2_{q+1}\delta_{q+1}]\\
     &=& \eta^2_{q+2} \delta_{q+2} + (1-\eta^2)(1-\eta^2_{q+2})[ \psi_{q+1}-\eta^2_{q+1}\delta_{q+1}]\\
     &=& \eta^2_{q+2} \delta_{q+2} + (1-\eta^2)(1-\eta^2_{q+2})\psi_{q+1}\\
     &=& \eta^2_{q+2} \delta_{q+2} + (1-\eta^2)(1-\eta^2_{q+2})\psi_{q+1}+\eta^2 \delta_{q+2} (1-\eta^2_{q+2})\\
    &=& \eta^2\delta_{q+2} +(1-\eta^2)\psi_{q+1} - (1-\eta^2)\eta_{q+2}^2(\psi_{q+1}-\delta_{q+2}).
\end{eqnarray*}
That is
\begin{align*}
    \psi_{q+2} &= \eta^2\delta_{q+2}+(1-\eta^2)\psi_{q+1} -(1-\eta^2)\eta_{q+2}^2(\psi_{q+1}-\delta_{q+2})\\
    &= \eta^2\delta_{q+2}+(1-\eta^2)\psi_{q+1} -(1-\eta^2)\eta_{q+2}^2(\tilde\psi_{q+1}-\delta_{q+2})\\
    &\quad -(1-\eta^2)\eta_{q+2}^2(\psi_{q+1}-\tilde\psi_{q+2}).
\end{align*}
Finally we write
\begin{equation*}
A=\eta^2 \tilde A +(1-\eta^2)A +\eta^2(A-\tilde A).
\end{equation*}
Putting these together, we reformulate \eqref{ttt} as
\begin{align*}
    &A-\frac{1}{2}\nabla v_0\otimes\nabla v_0-\sym\nabla w_0-\psi_{q+2}Id\\
    &=\left\{\begin{aligned}
      &\eta^2(\tilde A-\frac{1}{2} \nabla\tilde v\otimes\nabla\tilde v-\sym\nabla\tilde w-\delta_{q+2}Id)\\
      & +(1-\eta^2)(A-\frac{1}{2}\nabla V_q\otimes\nabla V_q-\sym\nabla W_q-\psi_{q+1}Id)\\
      & +(1-\eta^2)\eta_{q+2}^2(\tilde \psi_{q+1}-\delta_{q+2})Id
      \end{aligned}\right\}(*)\\
    &\quad + \eta^2(A-\tilde A)+(1-\eta^2)\eta_{q+2}^2(\psi_{q+1}-\tilde\psi_{q+1})Id\\
    &\quad -\sym[(\nabla \eta^2)(\tilde v-V_q)]\otimes[(\eta^2\nabla\tilde v+(1-\eta^2)\nabla V_q]\\
    &\quad -\frac{1}{2}[(\nabla \eta^2)(\tilde v-V_q)]\otimes[(\nabla \eta^2)(\tilde v-V_q)]\\
    &\quad -\sym(\nabla \eta^2)\otimes(\tilde w-W_q)
\end{align*}
Compared with the definitions of $\tilde D,D_q$ in \eqref{tildeD} \eqref{ccc}, we see
\begin{equation}\label{eq*}
    (*) = \tilde D + (1-\eta^2)D_q.
\end{equation}
As for \eqref{sss}, recall the definition of $\E_i$, then we see
\begin{equation*}
    \tilde D+\sum_{i=1}^{N_*}(\frac{1}{2}\nabla  {v}_{i-1}\otimes\nabla  {v}_{i-1}+\sym\nabla   {w}_{i-1})
    -(\frac{1}{2}\nabla  {v}_i\otimes\nabla  {v}_i+\sym\nabla   {w}_i) =\sum_{i=1}^{N_*}\E_i.
\end{equation*}
Therefore we conclude
\begin{align*}
    D_{q+1}&=\sum_{i=1}^{N_*}\E_i +(1-\eta^2)D_q\\
    &\quad +\eta^2(A-\tilde A)+(1-\eta^2)\eta_{q+2}^2(\psi_{q+1}-\tilde\psi_{q+1})Id\\
    &\quad -\sym[(\nabla \eta^2)(\tilde v-V_q)]\otimes[(\eta^2\nabla\tilde v+(1-\eta^2)\nabla V_q]\\
    &\quad -\frac{1}{2}[(\nabla \eta^2)(\tilde v-V_q)]\otimes[(\nabla \eta^2)(\tilde v-V_q)]\\
    &\quad -\sym(\nabla \eta^2)\otimes(\tilde w-W_q).
\end{align*}
We estimate term by term. First, because of \eqref{dddd}, there holds
\begin{equation*}
    \|\sum_{i=1}^{N_*}\E_i\|_0\lesssim\delta_{q+1}(\frac{\mu_0}{\mu_1}+\frac{\mu_1}{\mu_2}+...+\frac{\mu_{N_*-1}}{\mu_{N_*}}).
\end{equation*}
And by assumption \eqref{prop2.a4}
\begin{equation*}
    (1-\eta^2)D_q = 0.
\end{equation*}
Next, recall $A\in C^{2,\kappa}(\bar\Omega,\S)$, hence
\begin{equation*}
    \|\eta^2(A-\tilde A)\|_0\lesssim \|A\|_1 l\lesssim l.
\end{equation*}
Also recall $\|\psi_{q+1}\|_1\lesssim 1$, so
\begin{equation*}
\|(1-\eta^2)\eta_{q+2}^2(\psi_{q+1}-\tilde\psi_{q+1})Id\|_0\lesssim \|\psi_{q+1}\|_1 l\lesssim l.
\end{equation*}
Similarly, using \eqref{phil1}, \eqref{prop2.a1}-\eqref{prop2.a3} and $\|\nabla\eta\|_0\lesssim \delta_{q+1}^{-1}$ , we obtain
\begin{gather*}
    \|-\sym[(\nabla \eta^2)(\tilde v-V_q)]\otimes[(\eta^2\nabla\tilde v+(1-\eta^2)\nabla V_q]\|_0\\
    \lesssim \|\eta\|_1\|V_q\|_2 l^2\cdot\|V_q\|_1
    \lesssim \sqrt{K}\delta_{q+1}^{-\frac{1}{2}}l,\\
    \| -\frac{1}{2}[(\nabla \eta^2)(\tilde v-V_q)]\otimes[(\nabla \eta^2)(\tilde v-V_q)]\|_0\\
    \lesssim (\|\eta\|_1\|V_q\|_2 l^2)^2
    \lesssim \delta_{q+1}^{-1}l^2,\\
    \|-\sym(\nabla \eta^2)\otimes(\tilde w-W_q)\|
    \lesssim \|\eta\|_1\|W_q\|_2 l^2
    \lesssim \delta_{q+1}^{-\frac{1}{2}}l.
\end{gather*}
As a result
\begin{equation*}
     \|D_{q+1}\|_0\lesssim \delta_{q+1}^{-1}l^2+\sqrt{K}\delta_{q+1}^{-\frac{1}{2}}l+\delta_{q+1}(\frac{\mu_0}{\mu_1}+\frac{\mu_1}{\mu_2}+...+\frac{\mu_{N_*-1}}{\mu_{N_*}}).
\end{equation*}
We take
\begin{equation*}
    \mu_i:=\mu_0^{1-\frac{i}{N_*}}\mu_{N_*}^\frac{i}{N_*}
\end{equation*}
and require
\begin{equation}\label{ww1}
    \sqrt{K}\delta_{q+1}^{-\frac{1}{2}}l\le \frac{\sigma\delta_{q+2}}{C},\quad 
    \frac{\mu_0}{\mu_{N_*}}\le
    (\frac{\sigma\delta_{q+2}}{C\delta_{q+1}})^{N_*}
\end{equation}
so that \eqref{prop2.2} follows. This concludes the proof of Proposition \ref{prop2}.
\end{proof}

\textbf{Step 4 (Conclusion): $(V_q,W_q)\to(v,w)$.}
The new inequalities we need are \eqref{ineq999}
\[
l=\frac{\sigma\delta_{q+1}^\frac{1}{2}}{CK\delta_q^\frac{1}{2}\lambda_q}\le \frac{\delta_{q+2}}{C_1},
\]
and the first inequality in \eqref{ww1}
\begin{equation*}
    l=\frac{\sigma\delta_{q+1}^\frac{1}{2}}{CK\delta_q^\frac{1}{2}\lambda_q}\le \frac{\sigma}{C_2\sqrt{K}}\delta_{q+1}^{\frac{1}{2}}\delta_{q+2}.
\end{equation*}
These amount to
\begin{equation*}
    \frac{1}{2}-c\le -b^2+\frac{1}{b^q}\log_a C(\sigma, K).
\end{equation*}
Recall at the end of Section \ref{sec1}, we take $a>1$ sufficiently large, $b>1$ close to $1$ and $c>N_*+\frac{1}{2}\ge\frac{3}{2}$, so this can be achieved without breaking the other inequalities we need. The rest is the same as in Section \ref{sec1}. This finishes the whole proof of Theorem \ref{mainthm3}, which yields Theorem \ref{mainthm2} and Theorem \ref{mainthm}.


\section*{Acknowledgments}
The second named author is partially supported by CAS Project for Young Scientists in Basic Research, Grant No.YSBR-031.

\appendix

\small

\bibliographystyle{plainnat}
\bibliography{ref}

\end{document}